\documentclass[11pt,reqno]{amsart}

\usepackage{lineno,hyperref}

\modulolinenumbers[5]

\usepackage{amsmath}
\usepackage{amssymb}
\usepackage{mathrsfs}
\usepackage{amsthm}
\usepackage{bm}
\usepackage{graphicx}
\usepackage{xcolor}  
\usepackage{tikz} 
\usetikzlibrary{arrows,shapes,chains}  
\usepackage{booktabs}
\usepackage{float}
\usepackage{subfig}
\usepackage{geometry}
\usepackage{color}

\newtheorem{Def}{Definition}[section]
\newtheorem{lem}[Def]{Lemma}
\newtheorem{theo}[Def]{Theorem}
\newtheorem{pro}[Def]{Proposition}
\newtheorem{rem}[Def]{Remark}

\newtheorem{assum}{Assumption}

\usepackage{pgf}
\definecolor{Green}{RGB}{0,128,0}

\newcommand{\LL}{\langle}
\newcommand{\RR}{\rangle}
\newcommand{\OO}{\mathcal O}
\newcommand{\R}{\mathbb R}

\newcommand{\mcal}{\mathcal}

\newcommand{\mscr}{\mathscr}
\newcommand{\mbb}{\mathbb}
\newcommand{\mbf}{\mathbf}

\newcommand{\ud}{\mathrm d}

\newcommand{\PD}{\partial}
\numberwithin{equation}{section}
\allowdisplaybreaks \allowdisplaybreaks[4]

\begin{document}

\title[LDP]{Asymptotics of large deviations of finite difference method for stochastic Cahn--Hilliard equation}

\author{Diancong Jin}
\address{School of Mathematics and Statistics, Huazhong University of Science and Technology, Wuhan 430074, China;
	Hubei Key Laboratory of Engineering Modeling and Scientific Computing, Huazhong University of Science and Technology, Wuhan 430074, China}
\email{jindc@hust.edu.cn}

\author{Derui Sheng}
\address{LSEC, ICMSEC, Academy of Mathematics and Systems Science, Chinese Academy of Sciences, Beijing 100190, China}
\email{sdr@lsec.cc.ac.cn}

\thanks{This work is supported by National key R\&D Program of China (No.\ 2020YFA0713701), National Natural Science Foundation of China (No. 12201228), and the Fundamental Research Funds for the Central Universities 3004011142.
}

\keywords{
Large deviations rate function, Finite difference method, Convergence analysis, $\Gamma$-convergence, Stochastic Cahn--Hilliard equation.}

\begin{abstract}
In this work, we establish the Freidlin--Wentzell large deviations principle (LDP) of the stochastic Cahn--Hilliard equation with small noise, which implies the one-point LDP. Further, we give the one-point LDP of the spatial finite difference method (FDM) for the stochastic Cahn--Hilliard equation. Our main result is the convergence of the one-point large deviations rate function (LDRF) of the spatial FDM, which is about the asymptotical limit of a parametric variational problem. The main idea for proving the convergence of the LDRF of the spatial FDM is via the $\Gamma$-convergence of objective functions, which relies on the
qualitative analysis of skeleton equations of the original equation and the numerical method. In order to overcome the difficulty that the drift coefficient is not one-side Lipschitz, we use the equivalent characterization of the skeleton equation of the spatial FDM and the discrete interpolation inequality to obtain the uniform boundedness of the solution to the underlying skeleton equation. This plays an  important role in deriving the $\Gamma$-convergence of objective functions.
\end{abstract}

\maketitle

\textit{AMS subject classifications}:  60F10, 60H35, 49J45

\section{Introduction}
The stochastic Cahn–Hilliard equation, as an important  phase field model,  can
describe the complicated phase separation and coarsening phenomena in a melted alloy that is
quenched to a temperature at which only two different concentration phases can exist stably \cite{CH58,CH20,NCS84}. In order to  reflect the intrinsic properties of stochastic Cahn–Hilliard equations, the numerical study for them has been an indispensable tool. There have been lots of works studying the convergence of numerical methods for stochastic Cahn–Hilliard equations, including the  finite element method \cite{CCZZ18,EL92,FLZ20,FKLL18,KLM11,LM11,QW20,ZG18}, spectral Galerkin method \cite{CH20,CHS21,ZL22} and  finite difference method (FDM) \cite{Sheng1}.

In this paper, we consider the  stochastic Cahn--Hilliard equation on $\OO:=[0,\pi]$ driven by the space-time white noise:
\begin{equation}\label{SCH}
	\left\{\begin{split}
		&\frac{\partial u^\varepsilon}{\partial t}(t,x) +\Delta^2 u^\varepsilon(t, x)=\Delta b(u^\varepsilon(t, x))+\sqrt{\varepsilon}\sigma(u^\varepsilon(t, x))\dot{W}(t, x),\quad (t, x) \in(0, T] \times\OO,\\
		&u^\varepsilon(0, x)=u_{0}(x),\quad x \in\OO,\\
		&\frac{\PD u^\varepsilon}{\PD x}(t,x)=\frac{\PD (\Delta u^\varepsilon)}{\PD x}(t,x)=0, ~\text{on}~[0,T]\times\PD\OO.
	\end{split}\right.
\end{equation}
Here, $T>0$ is a given positive number, the small parameter $\varepsilon\in(0,1]$ denotes the noise intensity, $\Delta:=\frac{\PD^2}{\PD x^2}$ denotes the Laplacian,
and $\{W(t,x), (t,x)\in\OO_T\}$ is a Brownian sheet defined on a complete filtered probability space
$\big(\Omega,\mcal F,\{\mcal F_t\}_{t\in[0,T]},\mbf P\big)$
with $\{\mcal F_t\}_{t\in[0,T]}$ satisfying the usual conditions, where $\OO_T:=[0,T]\times\OO$. The nonlinearity $b(x)=x^3-x$ is the derivative of the double well potential $\frac{1}{4}(x^2-1)^2$. The conditions on $\sigma$ and $u_0$ will be specified on Section \ref{Sec2}. We refer readers to \cite{Cardon2001,CH22} for the well-posedness of \eqref{SCH}. In this work, we focus on the asymptotics of the exact solution  and  numerical solutions for \eqref{SCH} as  the noise intensity $\varepsilon$ tends to zero. It is observed that $u^\varepsilon$ will converges to its deterministic counterpart $u^0$ with $u^0$ being the solution to \eqref{SCH} with $\varepsilon=0$, for example in the mean-square sense. In many cases, one may be interested in the exact asymptotics of the probability 
$	\mbf P(|u^{\varepsilon}(t,x)-u^0(t,x)|> \delta)$  
for some $\delta>0$ and $(t,x)\in\OO_T$ (see e.g., \cite{FW}), which is usually characterized by the large deviations principle (LDP) of $\{u^{\varepsilon}(t,x)\}_{\varepsilon>0}$. 

The LDP is concerned with the exponential decay of probabilities of rare events, where the decay speed is characterized by the large deviations rate function (LDRF); see Definition \ref{LDPdef} for the rigorous definition of the LDP. Recently, the LDPs of sample paths of stochastic differential equations (SDEs) with small noise, also called the  Freidlin--Wentzell LDP, have received extensive attention (see e.g., \cite{Dupuis08,Dembo} and references therein), since the pioneering work of Freidlin and Wentzell \cite{FW}. For the stochastic Cahn--Hilliard equation \eqref{SCH}, \cite{LDPCH09} shows that $\{u^\varepsilon\}_{\varepsilon>0}$ satisfies an LDP in $\mbf C([0,T];\mbf L^p(\OO))$, $p\ge 4$. Further, \cite{LDPofCH} strengthens the above result and establishes the LDP of $\{u^\varepsilon\}_{\varepsilon>0}$ on $\mbf C^{\alpha,0}([0,T];\mbf L^p(\OO))$, where $\mbf C^{\alpha,0}([0,T];\mbf L^p(\OO))$ contains  temporally $\alpha$-H\"older continuous functions. We note that \cite{Cardon2001} shows that $u^\varepsilon$ has almost surely (a.s.) continuous trajectories provided $u_0$ is continuous. Thus, one natural question is whether the LDP on $\mbf C(\OO_T;\mbb R)$ holds for $\{u^\varepsilon\}_{\varepsilon>0}$.  For this end, the first result of this paper is to validate the LDP of $\{u^\varepsilon\}_{\varepsilon>0}$ on $\mbf C(\OO_T;\mbb R)$, which complements the existing results (see Theorem \ref{LDPofu}).

From the LDP of $\{u^\varepsilon\}_{\varepsilon>0}$ on $\mbf C(\OO_T;\mbb R)$, we immediately obtain that $\{u^{\varepsilon}(T,\bar{x})\}_{\varepsilon>0}$, $\bar{x}\in\OO$ also satisfies the LDP on $\mbb R$ with a good rate function $I:\mbb R\to[0,+\infty]$, which is called the one-point LDP of \eqref{SCH}; see Theorem \ref{LDPofu}. This means that for a Borel measurable set  $A\subseteq \mbb R$, the hitting probability
$\mbf P(u^{\varepsilon}(T, \bar{x})\in A)\asymp \exp\big({-\frac{1}{\varepsilon}\inf\limits_{x\in A} I(x)}\big)$ for sufficiently small $\varepsilon>0$. As an important quantity, the one-point
LDRF $I$ characterizes the exponential decay speed of the probability of $u^{\varepsilon}(T,\bar{x})$ deviating from its mean value $u^0(T,\bar{x})$. In terms of numerical simulation for the solution to \eqref{SCH}, a practical numerical method should preserve as many properties of \eqref{SCH} as possible, besides the convergence of itself. Thus, one natural problem is whether a numerical method can asymptotically preserve the exponential decay rate of the hitting probability $\mbf P(u^{\varepsilon}(T, \bar{x})\in A)$, when applied to \eqref{SCH}. More precisely, \emph{is there a numerical method satisfying the one-point LDP with a discrete rate function $I^n$  ($n$ being the discretization parameter) such that $I^n$ converges to $I$ in some sense?} Motivated by the above problem, this work focuses on giving the asymptotics of the one-point LDRF of the spatial FDM for \eqref{SCH}; see Section \ref{Sec3} for the construction of the spatial FDM. 

We show in Theorem \ref{LDPofun} that the spatial FDM for \eqref{SCH} satisfies the one-point LDP with a rate function $I^n$ via the Freidlin--Wentzell LDP of stochastic ordinary differential equations (SODEs) with small noise.  Notice that the one-point LDRFs $I$ and $I^n$ are implicitly determined by minimization problems (see \eqref{IandIn}):
\begin{align*}
	I(y)=\inf_{\{f\in \mathbf C(\OO_T;\mbb R)\}}J_y(f),\quad\ I^n(y)=\inf_{\{f\in \mathbf C(\OO_T;\mbb R)\}}J^n_y(f),\quad y\in\mbb R,
\end{align*}
where the objective function $J_y$ (resp. $J^n_y$) is the restriction of the  LDRF of sample paths of \eqref{SCH} (resp. the spatial FDM) on $\{f\in\mbf C(\OO_T;\mbb R):f(T,x_0)=y\}$.  $J_y$ (resp. $J^n_y$) is closely related to the solution mapping $\Upsilon$ (resp. $\Upsilon^n$) of the skeleton equation of \eqref{SCH}  (resp. the spatial FDM); see \eqref{eq:Skeleton} (resp. \eqref{DisSkeleton}) for the definition of $\Upsilon$ (resp. $\Upsilon^n$). The main purpose of this paper is to give the pointwise convergence of $I^n$ as $n\to+\infty$, which is about the asymptotical limit of minimization problems. We will use the technical  route  proposed in \cite{LDPofSWE} to analyze the pointwise convergence of $I^n$. Following the idea in \cite[Fig. 1]{LDPofSWE}, we  prove the pointwise of $I^n$ via the equi-coerciveness   and $\Gamma$-convergence of $\{J^n_y\}_{n\in\mbb N^+}$; see Appendix A for the basic introduction to $\Gamma$-convergence. Based on the technical  route in \cite{LDPofSWE}, the  convergence analysis of $I^n$ boils down to the
qualitative analysis of skeleton equations of the original equation and spatial FDM. 

First, based on the equi-continuity and uniform boundedness of $\Upsilon^n$ on bounded sets, we obtain the equi-coerciveness of $\{J^n_y\}_{n\in\mbb N^+}$ in Lemma \ref{equi}.  Then, we give the complete continuity of $\Upsilon$ and the
locally uniform convergence of $\Upsilon^n$ to $\Upsilon$, on basis of which we establish  the $\Gamma$-liminf inequality of  $\{J^n_y\}_{n\in\mbb N^+}$ in Lemma \ref{gammaLow}. Further, we drive the $\Gamma$-limsup inequality of  $\{J^n_y\}_{n\in\mbb N^+}$ in Lemma \ref{gammasup} by giving the locally Lipschitz property of $\Upsilon$ and invertibility of $\Upsilon^n$. Thus, the $\Gamma$-convergence of $\{J^n_y\}_{n\in\mbb N^+}$ comes from the $\Gamma$-liminf inequality and  $\Gamma$-limsup inequality, which combined with  the equi-continuity of $\{J^n_y\}_{n\in\mbb N^+}$ finally gives the pointwise convergence of $I^n$. The main difficulty brought about by the non-Lipschitz drift coefficient is the uniform boundedness of the skeleton equation of the spatial FDM. This is overcome by the discrete interpolation inequality and the equivalent characterization of the skeleton equation (see Lemma \ref{unbounded} and Appendix B). Different from \cite{LDPofSWE}, in the proof of the $\Gamma$-limsup inequality, we use the properties of discrete Neumann Laplacian $\Delta_{n}$ in Proposition \ref{Laplacen} to tackle the nonlinearity $\Delta_nb$. To sum up, in this work, we give the asymptotics of the one-point LDRF of the spatial FDM for stochastic Canhn--Hilliard equations with small noise. Our result indicates that the spatial FDM  can asymptotically preserve the exponential decay rate of the hitting probability $\mbf P(u^{\varepsilon}(T, \bar{x})\in A)$. 

The asymptotics of LDRFs of stochastic numerical methods has become an increasing active area. Recently, \cite{LDPosc} and \cite{LDPxde} show that a class of stochastic symplectic methods can (weakly)  asymptotically preserve the LDPs of the  underlying stochastic Hamiltonian systems, and reveal the superiority of stochastic symplectic methods from the perspective of the LDP. \cite{ChenCC} gives an error estimate between the one-point LDRF of the midpoint method and that of the linear stochastic Maxwell equations. Further, \cite{LDPofonepoint} gives the  convergence order of the one-point LDRF of the stochastic $\theta$-method for nonlinear SODEs with small noise, while \cite{LDPofSWE} analyzes the convergence of the one-point LDRF of the spatial FDM for stochastic wave equations with Lipschitz coefficients. We also refer readers to \cite{LDPofInvariant,LDPlan} for the study on the asymptotics for LDRFs of invariant measures of stochastic numerical methods. Our result first gives the convergence of one-point LDRFs of numerical methods for SPDEs with non-Lipschitz coefficients.

This article is organized as follows.  In Section \ref{Sec2}, we present some preliminaries  and the one-point LDP for \eqref{SCH}. Section \ref{Sec3} gives the one-point LDP of the spatial FDM. The pointwise convergence of the  one-point LDRF of the spatial FDM  is proved  in Section \ref{Sec:pointwise}. The appendix gives a basic introduction to $\Gamma$-convergence and a useful result.

\section{One-point LDP for stochastic Cahn--Hilliard equation}\label{Sec2}
In this section, we present the one-point LDP of the stochastic Cahn--Hilliard equation \eqref{SCH}. We begin with some notations and preliminaries on the LDP. Throughout this paper,  let $\mbb N^+$ be the set of all positive integers.  Denote by $|\cdot|$ the $2$-norm of a vector or matrix. For a given $U\subseteq \mbb R^d$, denote by $\mbf C(U;\mbb R^m)$ the space of   continuous functions from $U$ to $\mbb R^m$, endowed with the supremum norm $\|f\|_{\mbf C(U)}:=\sup_{x\in U}|f(x)|$, $f\in\mbf C(U;\mbb R^m)$. For $\alpha\in(0,1]$, denote by $\mbf C^{\alpha}(U;\mbb R^m)$ the space of all $\alpha$-H\"older continuous functions from $U$ to $\mbb R^m$, endowed  with the norm $\|f\|_{\mbf C^{\alpha}(U)}:=\|f\|_{\mbf C(U)}+[f]_{\mbf C^\alpha(U)}$, where the semi-norm
$[f]_{\mbf C^\alpha(U)}:=\sup\big\{\frac{|f(x)-f(y)|}{|x-y|^\alpha}:~x,\,y\in U,~x\neq y\big\}$. In addition, $\mbf C^{k}(U;\mbb R^m)$, $k\in\mbb N^+$, denotes the space of $k$th continuously differentiable functions  from $U$ to $\mbb R^m$, and
$\mbf C^{\infty}(U;\mbb R^m)=\cap_{k\in\mbb N^+}\mbf C^k(U;\mbb R^m)$. Let $\mbf L^p(U;\mbb R^m)$  stand for the space of measurable functions from $U$ to $\mbb R^m$ with finite norm $\|f\|_{\mbf L^p(U)}:=(\int_U |f(x)|^p\ud x)^{1/p}$  if $1\le p<+\infty$, and $\|f\|_{\mbf L^\infty(U)}:=\underset{x\in U}{\text{ess~sup}}~|f(x)|$ if $p=+\infty$.
Let $E$ and $F$ be given topological vector spaces. For a mapping $T:E\to F$, denote $T(K):=\{T(x): x\in K\}$, $K\subseteq E$, and especially denote $\text{Im}(T):=T(E)$. For a functional $T: E\to [0,+\infty]$, denote by $\mcal D_I:=\{x\in E:~ T(x)<+\infty\}$ the effective domain of $T$. The infimum of an empty set is always interpreted as $+\infty$. Let $K(a_1,a_2,\ldots,a_m)$ denote some generic constant dependent on the parameters $a_1,a_2,\ldots,a_m$, which may vary from one place to another.

Next, we give the definition of the LDP. Throughout this section, let $\mcal X$ be a \emph{Polish space}, i.e., complete and separable metric space. Below are the concepts of the rate function and LDP (see, e.g., \cite{Dembo}).
\begin{Def}\label{ratefun}
	A real-valued function $I:\mcal X\rightarrow[0,\infty]$ is called a rate function if it is lower semicontinuous, i.e., for each $a\in[0,\infty)$, the level set $I^{-1}([0,a])$ is a closed subset of $\mcal X$. If all level sets $I^{-1}([0,a])$, $a\in[0,\infty)$, are compact, then $I$ is called a good rate function.
\end{Def}

\begin{Def}\label{LDPdef}
	Let $I$ be a rate function and $\{\mu_\epsilon\}_{\epsilon>0}$ be a family of probability measures on  $\mcal X$. We say that $\{\mu_\epsilon\}_{\epsilon>0}$ satisfies an LDP on $\mcal X$ with the rate function $I$ if
	\begin{flalign}
		(\rm{LDP 1})\qquad \qquad&\liminf_{\epsilon\to 0}\epsilon\ln(\mu_\epsilon(U))\geq-\inf I(U)\qquad\text{for every open}~ U\subseteq \mcal X,\nonumber&\\
		(\rm{LDP 2})\qquad\qquad &\limsup_{\epsilon\to 0}\epsilon\ln(\mu_\epsilon(C))\leq-\inf I(C)\qquad\text{for every closed}~ C\subseteq \mcal X.&\nonumber
	\end{flalign}
\end{Def} 
\noindent We also say that a family of random variables  $\{Z_{\epsilon}\}_{\epsilon>0}$ valued on $\mcal X$  satisfies an LDP with the rate function $I$, if its distribution satisfies  (LDP1) and  (LDP2) in Definition \ref{LDPdef}. We refer readers to \cite{Dembo}  for more details about LDP.

\begin{pro}\cite[Theorem 4.2.1]{Dembo}\label{contraction}
	Let $\mcal Y$	be another Polish space,  $f:\mcal X\to\mcal Y$ be a continuous function, and $I: \mcal X\to [0,\infty]$ be a good rate function.
	\begin{itemize}
		\item[(a)] For each $y\in\mcal Y$, define
		\begin{align*}
			\tilde{I}(y)\triangleq\inf\left\{I(x):~x\in\mcal X,\quad y=f(x)\right\}.
		\end{align*}
		Then $\tilde{I}$ is a good rate function on $\mcal Y$.
		\item[(b)] If $I$ controls the LDP associated with a family of probability measures $\{\mu_{\epsilon}\}_{\epsilon>0}$ on $\mcal X$, then $\tilde{I}$ controls the LDP associated with  the family of probability measures $\left\{\mu_{\epsilon}\circ f^{-1}\right\}_{\epsilon>0}$ on $\mcal Y$.
	\end{itemize}
\end{pro}

In what follows, we will use the weak convergence method to 
give the Freidlin--Wentzell LDP of \eqref{SCH} and further
derive the one-point LDP of \eqref{SCH}. 
Our main assumptions on $\sigma$ and $u_0$ are as follows.

\begin{assum}\label{assum1}
	$\sigma:\mbb R\to\mbb R$ is bounded and globally Lipschitz continuous. 
\end{assum}
\begin{assum}\label{assum2}
	$u_0$ belongs to $\mbf C^1(\OO;\mbb R)$ and $u_0'(0)=u_0'(\pi)=0$.
\end{assum}
\begin{assum}\label{assum3}
	For any $x\in\mbb R$,	$\sigma(x)\neq 0$.
\end{assum}

The Green function related to $\partial_t+\Delta^2$ with the Neumann boundary condition  is  given by
$$G_t(x,y)=\sum_{j=0}^\infty e^{-\lambda_j^2t}\phi_j(x)\phi_j(y),\quad t\in[0,T],~ x,y\in\mathcal O,$$
where $\lambda_j=-j^2$, $\phi_j(x)=\sqrt{\frac{2}{\pi}}\cos(j x),$ $j\ge1$ and $\phi_0=\sqrt{\frac{1}{\pi}}$. 
Under Assumptions \ref{assum1}-\ref{assum2},
\eqref{SCH} admits a unique mild solution given by (see e.g., \cite[Theorem 1.3]{Cardon2001})
\begin{align*}
	u^{\varepsilon}(t,x)&=\int_\OO G_t(x,z)u_0(z)\ud z+\int_0^t\int_\OO \Delta G_{t-s}(x,z)b(u^{\varepsilon}(s,z))\ud z\ud s\\
	&\quad+\sqrt{\varepsilon}\int_0^t\int_\OO G_{t-s}(x,z)\sigma(u^{\varepsilon}(s,z))\ud W(s,z),\quad (t,x)\in\OO_T.
\end{align*}

The following proposition collects some basic properties of the Green function $G$.
\begin{pro}\label{Green}
	The following properties hold.	
	\begin{itemize} 
		\item [(1)] There is $C>0$ and $c>0$  such that for any $t\in(0,T]$,
		\begin{gather*}
			|G_t(x,y)|\le \frac{C}{t^{1/4}}\exp\Big(-c\frac{|x-y|^{4/3}}{t^{1/3}}\Big),\quad|\Delta G_t(x,y)|\le \frac{C}{t^{3/4}}\exp\Big(-c\frac{|x-y|^{4/3}}{t^{1/3}}\Big).
		\end{gather*}
	\end{itemize}
	\item [(2)] There is $C>0$ such that for any $v\in\mbf L^1(\OO_T;\mbb R)$ and $t\in[0,T]$,
	$$\Big\|\int_0^t\int_\OO \Delta G_{t-s}(\cdot,z)v(s,z)\ud z\ud s\Big\|_{\mbf L^\infty(\OO)}\le C\int_0^t(t-s)^{-3/4}\|v(s,\cdot)\|_{\mbf L^1(\OO)}\ud s.$$
	\item[(3)]  For any $\alpha\in(0,1)$,	there is $K(\alpha,T)>0$  such that for any $(t,x),(s,y)\in\OO_T$ with $t\ge s$,
	\begin{gather}
		\int_0^t\int_\OO|G_{t-r}(x,z)-G_{t-r}(y,z)|^2\ud z\ud r\le K(\alpha,T)|x-y|^2.\label{spaceHol}\\	 \int_0^s\int_\OO|G_{t-r}(x,z)-G_{s-r}(x,z)|^2\ud z\ud r+\int_s^t\int_\OO|G_{t-r}(x,z)|^2\ud z\ud r\le K(\alpha,T)|t-s|^{\frac{3\alpha}{4}}.\label{timeHol}
	\end{gather}
	\item[(4)]  For any $\alpha\in(0,1)$,	there is $K(\alpha,T)>0$  such that for any $(t,x),(s,y)\in\OO_T$ with $t\ge s$,
	\begin{gather}
		\int_0^t\int_\OO|\Delta G_{t-r}(x,z)-\Delta G_{t-r}(y,z)|\ud z\ud r\le K(\alpha,T)|x-y|.\label{spaceHol1}\\	 \int_0^s\int_\OO|\Delta G_{t-r}(x,z)-\Delta  G_{s-r}(x,z)|\ud z\ud r+\int_s^t\int_\OO|\Delta G_{t-r}(x,z)|\ud z\ud r\le K(\alpha,T)|t-s|^{\frac{3\alpha}{8}}.\label{timeHol1}
	\end{gather}
\end{pro}
\begin{proof}
	The first  property and third one come from Lemmas 1.2 and 1.8 of \cite{Cardon2001}, respectively. The second  property follows by applying  \cite[Eq. (1.12)]{Cardon2001} with $q=r=+\infty$ and $\rho=1$. The proofs of \eqref{spaceHol1} and \eqref{timeHol1} are  similar to those of (2.18), (2.19) and  (2.21) in \cite{Sheng}.    
\end{proof}

For any $M\ge0$, denote
$\mbb S_M:=\big\{\phi\in\mbf L^2(\OO_T;\mbb R):~\|\phi\|_{\mbf L^2(\OO_T)}\le M\big\}$ and $\mscr A_M:=\big\{\phi:\Omega\times\OO_T\to\mbb R,~\phi~\text{is}~\text{predictable and}~\phi(\omega)\in \mbb S_M,~\mbf P\text{-}a.s.\big\}$. In the sequel, we always endow $\mbb S_M$ with the weak topology on  $\mbf L^2(\OO_T;\mbb R)$ under which  $\mbb S_M$ is a compact Polish space. In order to prove the LDP of $\{u^\varepsilon\}_{\varepsilon>0}$ on $\mbf C(\OO_T;\mbb R)$ based on the weak convergence method, it suffices to study the asymptotics of the
controlled equation of \eqref{SCH}:
\begin{align}\label{controll}
	&u^{\varepsilon,v}(t,x)=\int_\OO G_t(x,z)u_0(z)\ud z+\int_0^t\int_\OO \Delta G_{t-s}(x,z)b(u^{\varepsilon,v}(s,z))\ud z\ud s\nonumber\\
	&+\sqrt{\varepsilon}\int_0^t\int_\OO G_{t-s}(x,z)\sigma(u^{\varepsilon,v}(s,z))\ud W(s,z)+\int_0^t\int_\OO G_{t-s}(x,z)\sigma(u^{\varepsilon,v}(s,z))v(s,z)\ud z\ud s
\end{align}
for $(t,x)\in\OO_T$ and $v\in\mscr A_M$.

Next, we present some  properties of the controlled process $\{u^{\varepsilon,v}\}$.
\begin{pro}\label{pro1}
	Suppose that  Assumptions \ref{assum1}-\ref{assum2} hold. Then for any $M\in(0,+\infty)$ and $p\ge 1$,
	\begin{align*}
		\sup_{\varepsilon\in(0,1)}\sup_{v\in\mscr A_M}\sup_{(t,x)\in\OO_T}\mbf E|u^{\varepsilon,v}(t,x)|^p<+\infty.
	\end{align*}
\end{pro}
\begin{proof}
	It has been shown in \cite[Proposition 3.2]{LDPofCH} that for any $p\ge 1$ and $q\ge p$,
	$$\sup_{\varepsilon\in(0,1)}\sup_{v\in\mscr A_M}\sup_{t\in[0,T]}\mbf E\|u^{\varepsilon,v}(t,\cdot)\|^q_{\mbf L^p(\OO)}<+\infty.$$
	Let $\varepsilon\in(0,1)$ and $v\in\mscr A_M$.	It follows from the above formula, Proposition \ref{Green} (2), $|b(x)|\le K(1+|x|^3)$, $x\in\mbb R$ and the Minkowski inequality that for any $p\ge 1$ and $t\in[0,T]$,
	\begin{align*}
		&\;	\sup_{x\in\OO}\mbf E\,\Big|\int_0^t\int_\OO\Delta G_{t-s}(x,z)b(u^{\varepsilon,v}(s,z))\ud z\ud s\Big|^p	
		\le \mbf E\,\Big\|\int_0^t\int_\OO\Delta G_{t-s}(\cdot,z)b(u^{\varepsilon,v}(s,z))\ud z\ud s\Big\|_{\mbf L^\infty(\OO)}^p\\
		\le &\; K\mbf E\,\Big|\int_0^t(t-s)^{-3/4}\|b(u^{\varepsilon,v}(s,\cdot))\|_{\mbf L^1(\OO)}\ud s\Big|^p
		\le  K\Big|\int_0^t(t-s)^{-3/4}\big(\mbf E\|b(u^{\varepsilon,v}(s,\cdot))\|_{\mbf L^1(\OO)}^p\big)^{1/p}\ud s\Big|^p\\
		\le &\; K(p)\Big|\int_0^t(t-s)^{-3/4}\big(1+\mbf E\|u^{\varepsilon,v}(s,\cdot)\|_{\mbf L^3(\OO)}^{3p}\big)^{1/p}\ud s\Big|^p\le K(p,M,T).
	\end{align*}
	By the Burkholder inequality, the boundedness of $\sigma$ and Proposition \ref{Green} (1),  it holds  that for any $p\ge 2$ and $(t,x)\in\OO_T$,
	\begin{align*}
		&\mbf E\big|\int_0^t\int_\OO G_{t-s}(x,z)\sigma(u^{\varepsilon,v}(s,z)\ud W(s,z)\big|^p
		\le  K(p)\mbf E\Big(\int_0^t\int_\OO|G_{t-s}(x,z)\sigma(u^{\varepsilon,v}(s,z)|^2\ud z\ud s\Big)^{p/2}\\
		\le &K(p)\Big(\int_0^t\int_\OO|G_{t-s}(x,z)|^2\ud z\ud s\Big)^{p/2}\le K(p,T).
	\end{align*}
	The boundedness of $\sigma$, $v\in\mscr A_M$, the H\"older inequality and Proposition \ref{Green} (1) give that for $p\ge 2$,
	\begin{align*}
		\mbf E\Big|\int_0^t\int_\OO G_{t-s}(x,z)\sigma(u^{\varepsilon,v}(s,z))v(s,z)\ud z\ud s\Big|^p\le K(p)\mbf E\Big(\int_0^t\int_\OO |G_{t-s}(x,z)|^2\ud z\ud s\|v\|_{\mbf L^2(\OO_T)}^2\Big)^{p/2}\le K(p,M,T).
	\end{align*}
	Further, applying \cite[Lemma 2.2]{Cardon2001} yields
	\begin{align*}
		\sup_{(t,x)\in\OO_T}\Big|\int_\OO G_{t}(x,z)u_0(z)\ud z\Big|\le  K(T).
	\end{align*}
	Finally, by the H\"older inequality and \eqref{controll}, we have  that for any $p\ge 2$, $\varepsilon\in(0,1)$, $v\in\mscr A_M$ and $(t,x)\in\OO_T$,
	\begin{align*}
		\mbf E|u^{\varepsilon,v}(t,x)|^p\le K(p,M,T).
	\end{align*}
	The above formula and the H\"older inequality immediately yield the conclusion for the case $p\in[1,2)$. 
\end{proof}

\begin{pro}\label{pro2}
	Let $\mcal B\subseteq \big\{\phi:\Omega\times \OO_T\to\mbb R:\phi~\text{is predictable and}~\|\phi\|_{\mbf L^2(\OO_T)}<+\infty,~\mbf P\text{-}a.s.\big\}$ be a family such that for all $p\ge 2$,
	$		\sup\limits_{f\in\mcal B}\sup\limits_{(t,x)\in\OO_T}\mbf E|f(t,x)|^p<+\infty.$
	For any $f\in\mcal B$ and $v\in\mscr A_M$, $M\in(0,+\infty)$, define 
	\begin{gather*}
		\Phi_1(t,x):=\int_0^t\int_\OO \Delta G_{t-r}(x,z)f(r,z)\ud z\ud r,\\
		\Phi_2(t,x):=\int_0^t\int_\OO G_{t-r}(x,z)f(r,z)\ud W(r,z),\\ \Phi_3(t,x):=\int_0^t\int_\OO G_{t-r}(x,z)f(r,z)v(r,z)\ud z\ud r
	\end{gather*}
	for	$(t,x)\in\OO_T$. Then for $\alpha\in(0,\frac{3}{8})$ and $i=1,2,3$, 
	$\sup\limits_{f\in\mcal B,v\in\mscr A_M}\mbf E\|\Phi_i\|_{\mbf C^{\alpha}(\OO_T)}<+\infty.$
\end{pro}
\begin{proof}
	We only prove the conclusion for $i=3$, the case $i=1,2$ is almost same by an  additional application of the Burkholder inequality. 
	It follows from the H\"older inequality, the Minkowski inequality, Proposition \ref{Green} (3) and $\|v\|_{\mbf L^2(\OO_T)}\le M$, $\mbf P$-a.s.\ that for any $s\le t\le T$, $x,y\in\OO$, $p\ge 1$ and $\alpha\in(0,1)$,
	\begin{align*}
		&\;\mbf E|\Phi_3(t,x)-\Phi_3(s,y)|^{2p}\\
		\le &\; K(p)\mbf E|\Phi_3(t,x)-\Phi_3(t,y)|^{2p}+K(p)\mbf E|\Phi_3(t,y)-\Phi_3(s,y)|^{2p}\\
		\le &\;K(p,M)\Bigg[\mbf E\Big(\int_0^t\int_\OO|G_{t-r}(x,z)-G_{t-r}(y,z)|^2|f(r,z)|^2\ud z\ud r\Big)^p+\mbf E\Big(\int_s^t\int_\OO|G_{t-r}(y,z)|^2|f(r,z)|^2\ud z\ud r\Big)^p\\
		&\;+\mbf E\Big(\int_0^s\int_\OO|G_{t-r}(y,z)-G_{s-r}(y,z)|^2|f(r,z)|^2\ud z\ud r\Big)^p\Bigg]\\
		\le&\; K(p,M)\Bigg[\Big(\int_0^t\int_\OO |G_{t-r}(x,z)-G_{t-r}(y,z)|^2(\mbf E|f(r,z)|^{2p})^{\frac{1}{p}}\ud z\ud r\Big)^p\\
		&\;+\Big(\int_s^t\int_\OO |G_{t-r}(y,z)|^2(\mbf E|f(r,z)|^{2p})^{\frac{1}{p}}\ud z\ud r\Big)^p+\Big(\int_0^s\int_\OO |G_{t-r}(x,z)-G_{s-r}(x,z)|^2(\mbf E|f(r,z)|^{2p})^{\frac{1}{p}}\ud z\ud r\Big)^p\Bigg]\\
		\le&\; K(\alpha,p,M)\big(|t-s|^{\frac{3\alpha p}{4}}+|x-y|^{2p}\big).
	\end{align*}
	Thus, we obtain that for any $p\ge 2$, $(t,x),(s,y)\in\OO_T$ and $\alpha\in(0,\frac{3}{8})$,
	\begin{align*}
		\sup_{f\in\mcal B,v\in\mscr A_M} (\mbf E|\Phi_3(t,x)-\Phi_3(s,y)|^p)^{1/p}\le K(\alpha,p,M)\big(|t-s|^{\alpha}+|x-y|^{\alpha}\big).
	\end{align*}
	Further, using  the Kolmogorov continuity theorem \cite[Theorem C.6]{KD14} yields that for any  $\alpha\in(0,\frac{3}{8})$, $p>\frac{2}{\alpha}$ and $q\in(0,1-\frac{2}{\alpha p})$,
	\begin{align}\label{kolmogorov}
		\sup_{f\in\mcal B,v\in\mscr A_M}\mbf E\Bigg[\sup_{\underset{(t,x)\neq (s,y)}{(t,x),(s,y)\in\OO_T}
		}\Big|\frac{\Phi_3(t,x)-\Phi_3(s,y)}{\rho((t,x),(s,y))^{\alpha q}}\Big|^p\Bigg]<+\infty,
	\end{align}
	where $\rho((t,x),(s,y))=(|t-s|^2+|x-y|^2)^{1/2}$.
	For any $\eta\in(0,\frac{3}{8})$, choose $\alpha\in(\eta,\frac{3}{8})$. Then taking sufficiently large $p$ and $q=\frac{\eta}{\alpha}$ in \eqref{kolmogorov}, we have
	$\sup_{f\in\mcal B,v\in\mscr A_M}\mbf E[\Phi_3]_{\mbf C^{\eta}(\OO_T)}<+\infty$ for any $\eta\in(0,\frac{3}{8})$.
	Again by \eqref{kolmogorov} and $\Phi_3(0,0)=0$, it holds that $\sup\limits_{f\in\mcal B,v\in\mscr A_M}\mbf E\|\Phi_3\|_{\mbf C(\OO_T)}<+\infty$. In this way, we complete the proof. 
\end{proof}

\begin{pro}\label{pro3}
	Let $Z^\varepsilon_f(t,x):=\sqrt{\varepsilon}\int_0^t\int_0^1 G_{t-s}(x,y)f(s,y)\ud W(s,y)$, $(t,x)\in\OO_T$ and $\mcal B$  be as in Proposition \ref{pro2}. Then for every family $\{f^\varepsilon\}_{\varepsilon>0}\subseteq \mcal B$, $Z^\varepsilon_{f^{\varepsilon}}$ converges to $0$ in $\mbf C(\OO_T;\mbb R)$ in probability as $\varepsilon\to 0$. 
\end{pro}
\begin{proof}
	By Proposition \ref{pro2}, $\mbf E\|Z^\varepsilon_{f^\varepsilon}\|_{\mbf C(\OO_T)}\le \sqrt{\varepsilon}\sup_{f\in\mcal B}\mbf E\|\Phi_2\|_{\mbf C(\OO_T)}\le K\sqrt{\varepsilon}$, which yields the desired result. 
\end{proof}

\begin{theo}\label{LDPofu}
	Let Assumptions \ref{assum1}-\ref{assum2} hold. Then $\{u^\varepsilon\}_{\varepsilon>0}$ satisfies the LDP  on $\mbf C(\OO_T;\mbb R)$ with the good rate function given by
	\begin{gather}\label{J}
		J(f):=\inf_{\{h\in \mbf L^2(\OO_T;\mbb R):~ \Upsilon(h)=f\}}\frac12\|h\|^2_{\mbf L^2(\OO_T)},\quad f\in\mbf C(\OO_T;\mbb R),
	\end{gather}
	where $\Upsilon$ is the solution mapping which takes  $h\in\mbf L^2(\OO_T;\mbb R)$ to the solution of the following skeleton equation
	\begin{align}\label{eq:Skeleton}
		f(t,x)=&\;\int_\OO G_t(x,z)u_0(z)\ud z+\int_0^t\int_\OO \Delta G_{t-s}(x,z)b(f(s,z))\ud z\ud s +\int_0^t\int_\OO G_{t-s}(x,z)\sigma(f(s,z))h(s,z)\ud z\ud s.
	\end{align}
	In addition, for any $\bar x\in\OO$,	$\{u^{\varepsilon}(T,\bar x)\}_{\varepsilon>0}$ satisfies the LDP on $\mbb R$ with the good rate function $I$ given by
	\begin{align*}
		I(y):=\inf_{\{f\in \mbf C(\OO_T;\mbb R):~f(T,\bar x)=y\}}J(f),\quad y\in\R.
	\end{align*}
\end{theo}
\begin{proof}
	This proof is similar to  that of \cite[Theorem 9]{Dupuis08} and we only give the sketch of the proof.
	Let $\theta:[0,1)\to[0,1)$ be a measurable map such that $\lim\limits_{r\to 0}\theta(r)=\theta(0)=0$.
	Then based on Propositions \ref{pro1}-\ref{pro3}, 
	one can use the same procedure as in the proof of \cite[Theorem 12]{Dupuis08} to prove that for any family $\{v^{\varepsilon}\}\subseteq \mscr A_M$    converging to some $v\in \mscr A_M$ in distribution as $\varepsilon\to 0$, $u^{\theta(\varepsilon),v^{\varepsilon}}$, $\varepsilon>0$ as $\mbf C(\OO_T;\mbb R)$-valued random variables converge to $u^{0,v}$ in distribution (see \eqref{controll} for the definition of $u^{\varepsilon,v}$). Further, applying the above results and \cite[Theorem 7]{Dupuis08}, we conclude that $\{u^{\varepsilon}\}_{\varepsilon>0}$ satisfies the LDP on $\mbf C(\OO_T;\mbb R)$ with the good rate function $J$ given by \eqref{J}. For fixed  $\bar x\in\OO$, define the coordinate map $\xi_{(T,\bar x)}: \mbf C(\OO_T;\mbb R)\rightarrow \R$ by $\xi_{(T,\bar x)}(f)=f(T,\bar x)$ for $f\in\mbf C(\OO_T;\mbb R)$. It follows from the continuity of $\xi_{(T,\bar x)}$ and the contraction principle (Proposition \ref{contraction}) that $\{u^{\varepsilon}(T,\bar x)\}_{\varepsilon>0}$ satisfies an LDP with the good rate function $I$. 
\end{proof}

\section{One-point LDP for finite difference method}\label{Sec3}
In this section, we derive the one-point LDP of the spatial FDM for \eqref{SCH}. We begin with the construction of the spatial FDM. We note that \cite{Sheng1} studies the density convergence of the spatial FDM for stochastic Cahn--Hilliard equations with the Dirichlet boundary condition.

Let $h=\frac{\pi}{n}$ be the uniform space mesh size for a positive integer $n$. Given a function $w$ defined on $\{x_1,\ldots,x_n\}$ with $x_k=\frac{2k-1}{2n}\pi$, $k=1,\ldots,n$, 
we define the difference operator $$\delta_n w_i:=\frac{w_{i-1}-2w_i+w_{i+1}}{h^2},~i=1,\ldots,n,$$
where $w_i:=w(x_i)$. Then one approximates $u^{\varepsilon}(t,x_k)$  by $\{u^{\varepsilon,n}(t,x_k)\}_{n\ge1}$:
\begin{align*}
	\begin{cases}
		\quad\ud u^{\varepsilon,n}(t,x_k)+ \delta_n^2u^{\varepsilon,n}(t,x_k)\ud t\\
		= \delta_n b(u^{\varepsilon,n}(t,x_k))\ud t+\frac{n}{\pi}\sqrt{\varepsilon}\sigma(u^{\varepsilon,n}(t,x_k))\ud(W(t,kh)-W(t,(k-1)h)),~ t\in(0,T],\\
		u^{\varepsilon,n}(0,x_k)=u_0(x_k)	\end{cases}
\end{align*}
for $k=1,2,\ldots,n$, under the boundary conditions $u^{\varepsilon,n}(t,x_{1})=u^{\varepsilon,n}(t,x_0),$ $u^{\varepsilon,n}(t,x_{n+1})=u^{\varepsilon,n}(t,x_n)$, 
$u^{\varepsilon,n}(t,x_{-1})=u^{\varepsilon,n}(t,x_2)$, and $u^{\varepsilon,n}(t,x_{n+2})=u^{\varepsilon,n}(t,x_{n-1})$.
For $x\in[x_k,x_{k+1}]$ with $k\in\{1,\ldots,n-1\}$, we use the polygonal interpolation
$$u^{\varepsilon,n}(t,x):=u^{\varepsilon,n}(t,x_k)+\frac{n}{\pi}(x-x_k)\left(u^{\varepsilon,n}(t,x_{k+1})-u^{\varepsilon,n}(t,x_k)\right),\quad k\in\{1,\ldots,n-1\},$$
and set $u^{\varepsilon,n}(t,x)=u^{\varepsilon,n}(t,x_1)$ if $x\in[0,x_1]$ and $u^{\varepsilon,n}(t,x)=u^{\varepsilon,n}(t,x_n)$ if  $x\in[x_n,\pi]$.

Introduce the notation
$U^{\varepsilon,n}_k(t):=u^{\varepsilon,n}(t,x_k)$ and $W^n_k(t):=\sqrt{\frac{n}{\pi}}\left(W(t,kh)-W(t,(k-1)h)\right)$, for $k=1,\ldots,n$,
and denote $U^{\varepsilon,n}(t)=(U^{\varepsilon,n}_1(t),\ldots,U^{\varepsilon,n}_{n}(t))^\top$ and $W^n(t)=(W^n_1(t),\ldots,W^n_{n}(t))^{\top}$.
Let \begin{align*}
	A_n:=\frac{n^2}{\pi^2}\left[\begin{array}{cccccc}-1 & 1 & 0 & \cdots & \cdots & 0 \\1 & -2 & 1 & \ddots & \ddots & \vdots \\0 & 1 & -2 & 1 & \ddots & \vdots \\\vdots & 0 & \ddots & \ddots & \ddots & 0 \\\vdots & \ddots & \ddots & 1 & -2 & 1 \\0 & \cdots & \cdots & 0 & 1 & -1\end{array}\right],
\end{align*}
which is a symmetric and nonpositive-definite matrix. For $j\in\{0,\ldots,n-1\}$, $e_j=(e_j(1),\ldots,e_j(n))^\top$ given by
\begin{equation}\label{ejk}
	e_j(k)
	=\sqrt{\frac{\pi}{n}}\phi_j(x_k),\quad k=1,\ldots,n,
\end{equation} 
is an eigenvector of $ A_n$ with the associated  eigenvalue $\lambda_{j,n}=-j^2c_{j,n}$, where $\frac{4}{\pi^2}\le c_{j,n}:=\sin^2(\frac{j}{2n}\pi)/(\frac{j}{2n}\pi)^2\le 1.$  
We also remark that $\{e_0,\ldots,e_{n-1}\}$ also forms an orthonormal basis of $\R^{n}$, i.e.,
\begin{align}\label{orth0}
	\langle e_i,e_j\rangle_{\mathbb R^n}=\frac{\pi}{n}\sum_{k=1}^n\phi_i(x_k)\phi_j(x_k)=\int_\OO\phi_i(\kappa_n(x))\phi_j(\kappa_n(x))\ud x=\delta_{ij},
\end{align} where $\kappa_n(x):=(2[\frac{xn}{\pi}]+1)\frac{\pi}{2n}$.

Then we obtain an $n$-dimensional SDE 
\begin{align}\label{NCH}
	\ud U^{\varepsilon,n}(t)+ A_n^2U^{\varepsilon,n}(t)\ud t=  A_nB_n(U^{\varepsilon,n}(t))\ud t+\sqrt{\frac{\varepsilon n}{\pi}}\Sigma_n(U^{\varepsilon,n}(t))\ud W^n(t),~t\in(0,T]
\end{align}
with the initial condition $U^{n}(0)=(u_0(x_1),\ldots,u_0(x_n))^\top$, where  $B_n(a):=(b(a_1),\ldots,b(a_n))^\top$ and $\Sigma_n(a):=\textrm{diag}(\sigma(a_1),\ldots,\sigma(a_n))$ for any $a=(a_1,\ldots,a_n)\in\mbb R^n$.
It is easy to see that for $x,y\in\R$,
\begin{align}\label{fbfa}
	(b(y)-b(x))(x-y)&\le(x-y)^2,\\\label{fb-fa}
	|b(x)-b(y)|&\le c_0(1+|x|^2+|y|^2)|x-y|. 
\end{align}   

The following result gives the well-posedness of \eqref{NCH}.
\begin{pro}
	Let  Assumptions \ref{assum1}-\ref{assum2} hold. Then for any given $n\ge 1$, $\varepsilon\in(0,1]$, the SDE \eqref{NCH} admits a unique solution $\{U^{\varepsilon,n}(t),t\in[0,T]\}$, which is an $a.s.$ continuous $\mathscr F_t$-adapted process such that $\mathbb P$-a.s. for all $t\in[0,T]$,
	\begin{align*}
		U^{\varepsilon,n}(t)=U^n(0)+\int_0^t\left(- A_n^2U^{\varepsilon,n}(s)+ A_nB_n(U^{\varepsilon,n}(s))\right)\ud s+\sqrt{\frac{\varepsilon n}{\pi}}\int_0^t\Sigma_n(U^{\varepsilon,n}(s))\ud W^n(s).
	\end{align*} 
\end{pro}
\begin{proof}  
	In view of \cite[Theorem 3.1.1]{PR07} and Assumption \ref{assum1}, it suffices to show that $\R^{n}\ni x\mapsto  A_nF_n(x)\in\R^{n}$ satisfies that for all $R\in(0,\infty)$, $x,y\in\R^{n}$ with $|x| ,|y| \le R$,
	\begin{align}\label{lwm}
		\langle x-y, A_nB_n(x)- A_nB_n(y)\rangle \le K_n(R)|x-y| ^2, \quad\text{(local weak monotonicity)}
	\end{align}
	and 
	\begin{align}\label{wc}
		\langle x, A_nB_n(x)\rangle \le K_n(1+|x| ^2), \quad\text{(weak coercivity)}
	\end{align}
	for some  constants $K_n(R)$ and $K_n$.   By definition, 
	\begin{align*}
		&\quad\frac{\pi^2}{n^2}\langle x-y, A_nB_n(x)- A_nB_n(y)\rangle \\
		&=\sum_{i=2}^{n-1}(b(x_{i+1})-b(y_{i+1})+2b(y_i)-2b(x_i)+b(x_{i-1})-b(y_{i-1}))(x_i-y_i)\\
		&\quad+(b(x_{2})-b(y_{2})+b(y_1)-b(x_1))(x_1-y_1)+(b(y_{n})-b(x_{n})+b(x_{n-1})-b(y_{n-1}))(x_{n}-y_{n}).
	\end{align*}
	Applying the Young inequality, and using \eqref{fbfa} and \eqref{fb-fa} give that for $x,y\in\R^{n}$ with $|x| ,|y| \le R$,
	\begin{align*}
		&\quad\frac{\pi^2}{n^2}\langle x-y, A_nB_n(x)- A_nB_n(y)\rangle \\
		\le&2\sum_{i=2}^{n-1}|x_i-y_i|^2+\frac{1}{2}\sum_{i=2}^{n-1}\left(|b(x_{i+1})-b(y_{i+1})|^2+|b(x_{i-1})-b(y_{i-1})|^2+2|x_i-y_i|^2\right)\\
		&+\frac{1}{2}\left(|b(x_{2})-b(y_{2})|^2+3|x_1-y_1|^2+|b(x_{n-1})-b(y_{n-1})|^2+3|x_{n}-y_{n}|^2\right)\\
		\le&3|x-y| ^2+\sum_{i=1}^{n}|b(x_{i})-b(y_{i})|^2
		\le3|x-y| ^2+c_0^2\sum_{i=1}^{n}(1+|x_i|^2+|y_i|^2)^2|x_{i}-y_{i}|^2\\
		\le &\left(c_0^2(2R^2+1)^2+3\right)|x-y| ^2,
	\end{align*}
	which proves the local weak monotonicity \eqref{lwm} with $K_n(R)=\left(c_0^2(2R^2+1)^2+3\right)\frac{n^2}{\pi^2}.$ Similarly, by the elementary relation $a^3b\le \frac{3}{4}a^4+\frac{1}{4}b^4$, one can also verify $\langle x, A_n(x_1^3,\ldots,x_n^3)^\top\rangle \le 0$. Hence,
	\begin{align}\label{xBf}
		&\langle x, A_nB_n(x)\rangle 
		=\langle x, A_n(x_1^3,\ldots,x_n^3)^\top\rangle -\langle x, A_nx\rangle 
		\le -\langle x, A_nx\rangle  .
	\end{align}
	Since $- A_n$ is a symmetric matrix, we have $|- A_n|=\max_{0\le j\le n-1}(-\lambda_{j,n})\le(n-1)^2$.
	By further using $ab\le\frac{1}{2}a^2+\frac{1}{2}b^2$ for $a,b\in\R,$ we obtain 
	$\langle x, A_nB_n(x)\rangle \le-\langle x, A_nx\rangle \le (n-1)^2|x|^2$,
	i.e., the weak coercivity condition \eqref{wc} with $K_n=(n-1)^2$. Thus, the proof is complete.  
\end{proof}

Next, we give the equation satisfied by $u^{\varepsilon,n}$. First, we introduce the interpolation operator $\Pi_n$. For any measurable function $w:\OO\to\mbb R$, the function $\Pi_n(w)$ is given by
\begin{align*}
	\Pi_n(w)(x)=
	\begin{cases}
		w(x_1)\qquad&\text{if}~x\in[0,x_1],\\
		w(x_k)+\frac{n}{\pi}(w(x_{k+1})-w(x_k))(x-x_k)\qquad&\text{if}~x\in[x_k,x_{k+1}],~k=1,2,\ldots,n-1,\\
		w(x_n)\qquad&\text{if}~x\in[x_n,\pi].
	\end{cases} 
\end{align*}
Then let us introduce the discrete Green function
\begin{align*}
	G^n_t(x,y)=\sum_{j=0}^{n-1}\exp(-\lambda_{j,n}^2t)\phi_{j,n}(x)\phi_j(\kappa_n(y)), ~t\in[0,T],~x,y\in\OO,
\end{align*}
where $\phi_{j,n}=\Pi_n(\phi_j)$. Further, define the discrete Neumann Laplacian $\Delta_n$ by 
\begin{align*}
	(\Delta_n w)(x)= 
	\begin{cases}
		\frac{n^2}{\pi^2}\big(w(x_2)-w(x_1)\big)\qquad&\text{if}~x\in[0,\frac{\pi}{n}),
		\vspace{2mm}
		\\
		\frac{n^2}{\pi^2}\big[w(\kappa_n(x)+\frac{\pi}{n})-2w(\kappa_n(x))+w(\kappa_n(x)-\frac{\pi}{n})\big]\qquad&\text{if}~x\in[\frac{\pi}{n},\frac{(n-1)\pi}{n}),
		\vspace{2mm}
		\\
		\frac{n^2}{\pi^2}\big(w(x_{n-1})-w(x_n)\big)\qquad&\text{if}~x\in[\frac{(n-1)\pi}{n},\pi]
	\end{cases}
\end{align*}
for any measurable function $w:\OO\rightarrow \R$. Then one immediately has $\Delta_n w(x)=\Delta _n w(\kappa_n(x))$, $x\in\OO$. In addition, a direct computation leads to 
\begin{align}\label{Dnphi}
	\Delta_n \varphi_j(\kappa_n(y))=\lambda_{j,n}\varphi_j(\kappa_n(y))
\end{align}
for $y\in\OO$, $j=1,\ldots,n-1$, which produces
\begin{align*}
	\Delta_{n,y}G^n_{t}(x,y)=\sum_{j=0}^{n-1}\lambda_{j,n}\exp(-\lambda_{j,n}^2t)\phi_{j,n}(x)\phi_j(\kappa_n(y)),
\end{align*}
where  $\Delta_{n,y}$ means that the operator $\Delta_n$ is posed w.r.t.\ the variable  $y$.

Using the variation of constant formula, we derive from \eqref{NCH} that
\begin{align}\label{Unt}
	U^{\varepsilon,n}(t)&=e^{- A_n^2t}U^n(0)+\int_0^t A_ne^{-A_n^2(t-s)}B_n(U^{\varepsilon,n}(s))\ud s+\sqrt{\frac{\varepsilon n}{\pi}}\int_0^te^{-A_n^2(t-s)}\Sigma_n(U^{\varepsilon,n}(s))\ud W^n(s).
\end{align}
Similar to \cite{GI98}, based on \eqref{Unt}, 
the relation $u^{\varepsilon,n}(t,x_k)=\sum_{j=0}^{n-1}\langle U^{\varepsilon,n}(t),e_j\rangle e_j(k)$, $k=1,\ldots,n$ and $A_ne_j=\lambda_{j,n}e_j$, $j=0,1,\ldots,n-1$, one  has 
\begin{align}\label{untx}\notag
	u^{\varepsilon,n}(t,x)=&\int_{\mathcal O}G^n_t(x,y)u_0(\kappa_n(y))\ud y+\int_0^t\int_{\mathcal O} \Delta_{n,y}G^n_{t-s}(x,y)b(u^{\varepsilon,n}(s,\kappa_n(y)))\ud y\ud s\\
	&+\sqrt{\varepsilon}\int_0^t\int_{\mathcal O}G^n_{t-s}(x,y)\sigma(u^{\varepsilon,n}(s,\kappa_n(y)))\ud W(s,y).
\end{align}

Next, we use the Freidlin--Wentzell LDP of \eqref{Unt} and the contraction principle (Proposition \ref{contraction}) to establish the one-point LDP of \eqref{untx}.
\begin{theo}\label{LDPofun}
	Let Assumptions \ref{assum1}-\ref{assum2} hold.	Then for any $\bar x\in\OO$ and $n\in\mbb N^+$,	$\{u^{\varepsilon,n}(T,\bar x)\}_{\varepsilon>0}$ satisfies the LDP on $\mbb R$ with the good rate function $I^n$ given by
	\begin{align*}
		I^n(y):=\inf_{\{f\in \mbf C(\OO_T;\mbb R):~f(T,\bar x)=y\}}J^n(f),\quad y\in\R.
	\end{align*}
	Here $J^n:\mbf C(\OO_T;\mbb R)\to \mbb R$ is defined by
	$$J^n(f):=\inf_{\{h\in \mbf L^2(\OO_T;\mbb R):~ \Upsilon^n(h)=f\}}\frac12\|h\|^2_{\mbf L^2(\OO_T)},\quad f\in\mbf C(\OO_T;\mbb R),$$
	where $\Upsilon^n$ is the solution mapping which takes  $h\in\mbf L^2(\OO_T;\mbb R)$ to the solution of the discrete skeleton equation 
	\begin{align}\label{DisSkeleton}
		f(t,x)
		=&\;\int_\OO G^n_t(x,y)u_0(\kappa_n(y))\ud y
		+\int_0^t\int_\OO \Delta_{n,y}G^n_{t-s}(x,y)b(f(s,\kappa_n(y)))\ud y\ud s\\ 
		&\;+\int_0^t\int_\OO G^n_{t-s}(x,y)\sigma(f(s,\kappa_n(y)))h(s,y)\ud y\ud s.\nonumber
	\end{align}
\end{theo}
\begin{proof}
	For given $n\in\mbb N^+$,  $\{U^{\varepsilon,n}\}_{\varepsilon>0}$ is viewed as a family of $\mbf C([0,T];\mbb R^n)$-valued random variables. Since \eqref{lwm} and \eqref{wc}
	hold, the coefficients of the SODE \eqref{NCH} satisfy the conditions of  \cite[Theorem 1.2]{LDPnonLip}. Thus, for any given $n\in\mbb N^+$, $\{U^{\varepsilon,n}\}_{\varepsilon>0}$ satisfies the LDP on $\mbf C([0,T];\mbb R^n)$ with the good rate function 
	\begin{align*}
		\mbb K^n(p):=\inf_{\{q\in \mbf L^2(0,T;\mbb R^n):~ \mcal A(q)=p\}}\frac{1}{2}\int_0^T|q(t)|^2\ud t,\quad p\in\mbf C([0,T];\mbb R^n),
	\end{align*}
	where for any $q\in \mbf L^2(0,T;\mbb R^n)$, $\mcal A(q)$ solves the equation
	\begin{align*}
		p(t)=U^n(0)-\int_0^tA_n^2p(s)\ud s+\int_0^tA_nB_n(p(s))\ud s+\sqrt{\frac{n}{\pi}}\int_0^t\Sigma_n(p(s))q(s)\ud s,\quad t\in[0,T].
	\end{align*}
	Denote $\mcal P_n(\OO_T;\mbb R):=\big\{w\in\mbf C(\OO_T;\mbb R):~w(t,\cdot)=\Pi_n(w(t,\cdot)),~t\in[0,T]\big\}$, endowed with the norm of $\mbf C(\OO_T;\mbb R)$, and define the operator $\tau:\mbf C([0,T];\mbb R^n)\to \mcal P_n(\OO_T;\mbb R)$ as follows: for any $a(\cdot)=(a_1(\cdot),a_2(\cdot),\ldots,a_n(\cdot))^\top\in\mbf C([0,T];\mbb R^n)$, 
	\begin{align*}
		\tau(a)(t,x)=\begin{cases}
			a_1(t)\qquad&\text{if}~x\in[0,x_1],\\
			a_k(t)+\frac{n}{\pi}(a_{k+1}(t)-a_k(t))(x-x_k)\qquad&\text{if}~x\in[x_k,x_{k+1}],~k=1,2,\ldots,n-1,\\
			a_n(t)\qquad&\text{if}~x\in[x_n,\pi],
		\end{cases}
	\end{align*}
	for any $t\in[0,T]$.
	Then one immediately has $u^{\varepsilon,n}=\tau(U^{\varepsilon,n})$. It follows from the continuity of $\tau$ and  Proposition \ref{contraction} that for any $n\in\mbb N^+$, $\{u^{\varepsilon,n}\}_{\varepsilon>0}$ satisfies the LDP in $\mcal P_n(\OO_T;\mbb R)$ with the good rate function
	\begin{align*}
		\tilde{J}^n(f)=\inf_{\{p\in\mbf C([0,T];\mbb R^n):~\tau(p)=f\}}\mbb K^n(p)=\mbb K^n(\tau^{-1}(f)),~f\in\mcal P_n(\OO_T;\mbb R).
	\end{align*}
	Further, for any $f\in\mcal P_n(\OO_T;\mbb R)$, we have $\tau^{-1}(f)=\vec{f}$ with $\vec{f}(t):=\big(f(t,x_1),f(t,x_2),\ldots,f(t,x_n)\big)^\top$. Thus,
	\begin{align*}
		\tilde{J}^n(f)=\mbb K^n(\vec{f})=\inf_{\{q\in \mbf L^2(0,T;\mbb R^n):~ \mcal A(q)=\vec{f}\}}\frac{1}{2}\int_0^T|q(t)|^2\ud t,~f\in\mcal P_n(\OO_T;\mbb R).
	\end{align*}

	Define the operator $\theta:\mbf L^2(0,T;\mbb R^n)\to \mcal N_n(\OO_T;\mbb R):=\big\{h\in\mbf L^2(\OO_T;\mbb R):~h(t,x)=h(t,\kappa_n(x)),~(t,x)\in[0,T]\big\}$ which maps $q(\cdot)=(q_1(\cdot),\ldots,q_n(\cdot))^\top\in\mbf L^2(0,T;\mbb R^n)$ as
	\begin{align}\label{theta}
		\theta(q)(t,x):=\sqrt{\frac{n}{\pi}}q_k(t),~\text{if}~x\in\big[\frac{k-1}{n}\pi,\frac{k}{n}\pi\big),~k=1,2,\ldots,n.
	\end{align}
	Next, we proves $\mcal A(q)=\vec{f}$ if and only if $\Upsilon^n(\theta(q))=f$.
	In fact,  $\mcal A(q)=\vec{f}$ if and only if 
	\begin{align}\label{vecf}
		\vec{f}(t)&=e^{-A_n^2t}U^n(0)+\int_0^t A_ne^{-A_n^2(t-s)}B_n(\vec{f}(s))\ud s+\sqrt{\frac{n}{\pi}}\int_0^te^{-A_n^2(t-s)}\Sigma_n(\vec{f}(s))q(s)\ud s, ~t\in[0,T].
	\end{align}
	Since $\{e_j\}_{j=0}^{n-1}$ forms the orthonormal basis of $\mbb R^n$, from \eqref{ejk}, $A_ne_j=\lambda_{j,n}e_j$ and the definition of $G^n_t(x,y)$, we deduce 
	\begin{align}\label{vecf1}
		\big(e^{-A_n^2t}U^n(0)\big)_k&=\sum_{j=0}^{n-1}e^{-\lambda_{j,n}^2t}\langle U^n(0),e_j\rangle e_j(k)=\sum_{j=0}^{n-1}e^{-\lambda_{j,n}^2t}\sum_{l=1}^n U^n_l(0)e_j(l) e_j(k)\nonumber\\
		&=\frac{\pi}{n}\sum_{l=1}^n\sum_{j=0}^{n-1}e^{-\lambda_{j,n}^2t}\phi_j(x_k)\phi_j(x_l)u_0(x_l)=\frac{\pi}{n}\sum_{l=1}^n G^n_t(x_k,x_l)u_0(x_l)\nonumber\\
		&=\int_\OO G^n_t(x_k,\kappa_n(y))u_0(\kappa_n(y))\ud y.
	\end{align}
	Similarly, one has
	\begin{align}\label{vecf2}
		\big(A_ne^{-A_n^2(t-s)}B_n(\vec{f}(s))\big)_k
		=\int_\OO \Delta_{n,y}G^n_{t-s}(x_k,\kappa_n(y))b(f(s,\kappa_n(y)))\ud y,
	\end{align}
	and uses \eqref{theta} to get
	\begin{align}\label{vecf3}
		\Big(\sqrt{\frac{n}{\pi}}e^{-A_n^2(t-s)}\Sigma_n(\vec{f}(s))q(s)\Big)_k=\int_\OO G^n_{t-s}(x_k,\kappa_n(y))\sigma(f(s,\kappa_n(y)))\theta(q)(s,\kappa_n(y))\ud y.
	\end{align}
	Plugging \eqref{vecf1}-\eqref{vecf3} into \eqref{vecf} gives
	\begin{align*}
		f(t,x_k)
		=&\;\int_\OO G^n_t(x_k,y)u_0(\kappa_n(y))\ud y+
		\;+\int_0^t\int_\OO \Delta_{n,y}G^n_{t-s}(x_k,y)b(f(s,\kappa_n(y)))\ud y\ud s\\ 
		&\;+\int_0^t\int_\OO G^n_{t-s}(x_k,y)\sigma(f(s,\kappa_n(y)))\theta(q)(s,y)\ud y\ud s.
	\end{align*}
	Thus, $f=\Upsilon^n(\theta(q))$ due to $f(t,\cdot)=\Pi_n(f(t,\cdot))$ and $\Pi_n(G^n_t(\cdot,y))=G^n_t(\cdot,y)$. In this way, we obtain
	\begin{align*}
		\tilde{J}^n(f)=\inf_{\{q\in \mbf L^2(0,T;\mbb R^n):~ \Upsilon^n(\theta(q))=f\}}\frac{1}{2}\int_0^T|q(t)|^2\ud t,~f\in\mcal P_n(\OO_T;\mbb R).
	\end{align*}
	Notice that 
	\begin{align}\label{thetanorm}
		\int_{0}^{T}|q(t)|^2\ud t=\sum_{l=1}^n\int_0^T|q_l(t)|^2\ud t=\frac{\pi}{n}\int_0^T\sum_{l=1}^n|\theta(q)(t,x_l)|^2\ud t=\int_0^T\int_\OO|\theta(q)(t,x)|^2\ud x\ud t,
	\end{align}
	which, along with the fact that $\theta$ is a bijection  from $\mbf L^2(0,T;\mbb R^n)$ to $\mcal N_n(\OO_T;\mbb R)$, yields
	\begin{align*}
		\tilde{J}^n(f)=\inf_{\{h\in \mcal N_n(\OO_T;\mbb R):~ \Upsilon^n(h)=f\}}\frac{1}{2}\|h\|_{\mbf L^2(\OO_T)}^2,~f\in\mcal P_n(\OO_T;\mbb R).
	\end{align*}
	For fixed  $\bar x\in\OO$, define the coordinate map $\xi_{(T,\bar x)}: \mcal P_n(\OO_T;\mbb R)\rightarrow \R$ by $\xi_{(T,\bar x)}(f)=f(T,\bar x)$ for $f\in \mcal P_n(\OO_T;\mbb R)$. It follows from the continuity of $\xi_{(T,\bar x)}$, the LDP of $\{u^{\varepsilon,n}\}_{\varepsilon>0}$ on $\mcal P_n(\OO_T;\mbb R)$ and Proposition \ref{contraction} that $\{u^{\varepsilon}(T,\bar x)\}_{\varepsilon>0}$ satisfies an LDP with the good rate function 
	\begin{align*}
		I^n(y)=\inf_{\{f\in \mcal P_n(\OO_T;\mbb R):~  f(T,\bar{x})=y\}}\tilde{J}^n(f),~y\in\mbb R.
	\end{align*}
	Define $J^n:\mbf C(\OO_T;\mbb R)\to\mbb R$  by
	\begin{align*}
		J^n(f)=\inf_{\{h\in \mcal N_n(\OO_T;\mbb R):~  \Upsilon^n(h)=f\}}\frac{1}{2}\|h\|_{\mbf L^2(\OO_T)}^2, ~f\in\mbf C(\OO_T;\mbb R).
	\end{align*}
	Noting  that $\text{Im}(\Upsilon^n)\subseteq\mcal P_n(\OO_T;\mbb R)$ and using the convention $\inf\emptyset=+\infty$, we have
	\begin{align*}
		J^n(f)=
		\begin{cases}
			\tilde{J}^n(f)\qquad&\text{if}~f\in\mcal P_n(\OO_T;\mbb R),\\
			+\infty\qquad&\text{if}~f\in\mbf C(\OO_T;\mbb R)-\mcal P_n(\OO_T;\mbb R).
		\end{cases}
	\end{align*}
	Thus, 
	$	I^n(y)=\inf\limits_{\{f\in \mbf C(\OO_T;\mbb R):~  f(T,\bar{x})=y\}}J^n(f),~y\in\mbb R.$
	
	Finally, we give the equivalent representation of $J^n$:
	\begin{align}\label{JJn}
		J^n(f)=\inf_{\{h\in \mbf L^2(\OO_T;\mbb R):~  \Upsilon^n(h)=f\}}\frac{1}{2}\|h\|_{\mbf L^2(\OO_T)}^2, ~f\in\mbf C(\OO_T;\mbb R).
	\end{align}
	It suffices to show that 
	$$\inf\limits_{\{h\in\mbf L^2(\OO_T;\mbb R):~\Upsilon^n(h)=f\}}\frac{1}{2}\|h\|_{\mbf L^2(\OO_T)}^2=\inf\limits_{\{h\in\mcal N_n(\OO_T;\mbb R):~\Upsilon^n(h)=f\}}\frac{1}{2}\|h\|_{\mbf L^2(\OO_T)}^2$$ 
	provided that   $f\in \text{Im}(\Upsilon^n)$.  It is not difficult to see that
	$$\inf\limits_{\{h\in\mbf L^2(\OO_T;\mbb R):~\Upsilon^n(h)=f\}}\frac{1}{2}\|h\|_{\mbf L^2(\OO_T)}^2\le\inf\limits_{\{h\in\mcal N_n(\OO_T;\mbb R):~\Upsilon^n(h)=f\}}\frac{1}{2}\|h\|_{\mbf L^2(\OO_T)}^2$$ 
	provided that  $f\in \text{Im}(\Upsilon^n)$. 
	It remains to prove the inverse inequality. 
	Suppose that $f\in \text{Im}(\Upsilon^n)$, $h\in\mbf L^2(\OO_T;\mbb R)$, and $\Upsilon^n(h)=f$. Then 
	\begin{align}\label{sdr1}
		f(t,x)
		=	&\;\int_\OO G^n_t(x,y)u_0(\kappa_n(y))\ud y
		+\int_0^t\int_\OO \Delta_{n,y} G^n_{t-s}(x,y)b(f(s,\kappa_n(y)))\ud y\ud s\nonumber \\
		&\;	+\int_0^t\int_\OO G^n_{t-s}(x,y)\sigma(f(s,\kappa_n(y)))h(s,y)\ud y\ud s.
	\end{align}
	Further, we define $\tilde{h}\in\mcal N_n(\OO_T;\mbb R)$ as
	\begin{align}\label{htilde}
		\tilde{h}(t,x)=\frac{n}{\pi}\int_{\frac{(k-1)\pi}{n}}^{\frac{k\pi}{n}}h(t,y)\ud y,\quad \text{for}~t\in[0,T]~\text{and}~x\in\big[\frac{(k-1)\pi}{n},\frac{k\pi}{n}\big),~k=1,\ldots,n.
	\end{align}
	It is verified that $\Upsilon^n(\tilde{h})=f$ and $\|\tilde{h}\|_{\mbf L^2(\OO_T)}\le \|h\|_{\mbf L^2(\OO_T)}$ by the H\"older inequality.
	Thus,  for any $h\in\mbf L^2(\OO_T;\mbb R)$ satisfying $\Upsilon^n(h)=f$, there is some $\tilde{h}\in\mcal N_n(\OO_T;\mbb R)$ such that $\Upsilon^n(\tilde h)=f$ and $\|\tilde h\|^2_{\mbf L^2(\OO_T)}\le \| h\|_{\mbf L^2(\OO_T)}^2$. These imply 
	$$\inf\limits_{\{h\in\mbf L^2(\OO_T;\mbb R):~\Upsilon^n(h)=f\}}\frac{1}{2}\|h\|_{\mbf L^2(\OO_T)}^2\ge\inf\limits_{\{h\in\mcal N_n(\OO_T;\mbb R):~\Upsilon^n(h)=f\}}\frac{1}{2}\|h\|_{\mbf L^2(\OO_T)}^2.$$ 
	Thus \eqref{JJn} holds	and the proof is finished. 
\end{proof}

\section{Pointwise convergence of one-point LDRF of spatial FDM}\label{Sec:pointwise}
In this section, we prove the pointwise convergence of the one-point LDRF $I^n$ for the spatial FDM based on the technical route in \cite{LDPofSWE}, which is shown as follows. 
\begin{figure}[h]  
	\flushleft
	\tiny  
	\tikzstyle{format}=[rectangle,draw,thin,fill=white]  
	\tikzstyle{test}=[diamond,aspect=2,draw,thin]  
	\tikzstyle{point}=[coordinate,on grid,]  
	\begin{tikzpicture}[node distance=30mm,auto,>=latex',thin,start chain=going right,every join/.style={norm},] 
		\phantom{\node[format](blank){aaa};}
		\phantom{\node[format,right of=blank,node distance=35mm](blank0){aaa};}
		\node[format,below of=blank,node distance=30mm](point){$\overset{\text{Pointwise Convergence}}{\text{of}~I^n}$};
		\node[format,below of=blank0,node distance=15mm](gammaCov){$\Gamma$-convergence of $J^n_y$};
		\node[format,below of=gammaCov,node distance=30mm](coercive){ Equi-coerciveness of $J^n_y$};
		\node[format,right of=coercive,node distance=70mm](boundHolder){ $\overset{\text{Equi-continuity and Uniform Boundedness}}{\text{of}~ \Upsilon^n(\mbb S_a)}$};
		\phantom{\node[format,right of=blank0,node distance=35mm](blank1){aaa};}
		\node[format,below of=blank1,node distance=7mm](gammalow){$\Gamma$-liminf Inequality};
		\node[format,below of=gammalow,node distance=15mm](gammasup){$\Gamma$-limsup Inequality};
		\node[format,right of=blank1,node distance=45mm](compact){\phantom{hahjh}Compactness of $\Upsilon$ \phantom{hahjh}};
		\node[format,below of=compact,node distance=10mm] (localconver){
			$\overset{\text{	Locally Uniform Convergence}}{\text{of}~ \Upsilon^n~ \text{to}~ \Upsilon}$};
		\node[format,below of=localconver,node distance=10mm] (localLip){Locally Lipschitz Property of $\Upsilon$};
		\node[format,below of=localLip,node distance=10mm] (inverse){
			$\overset{\text{ Properties of Inverse Operator}}{\text{of}~ \Upsilon^n}$
		};
		\node[point,right of=gammaCov,node distance=15mm](point1){};
		\node[point,left of=gammalow,node distance=16mm](point2){};
		\node[point,left of=gammasup,node distance=16mm](point3){};
		\draw[->](point1)--(gammaCov);
		\draw[-](point1)-|(point2);
		\draw[-](gammalow)--(point2);
		\draw[-](gammasup)--(point3);
		\draw[-](point1)-|(point3);
		\node[point,right of=gammalow,node distance=12.5mm](point4){};
		\node[point,right of=gammasup,node distance=13mm](point5){};
		\node[point,left of=compact,node distance=19.8mm](point6){};
		\node[point,left of=localconver,node distance=17mm](point7){};
		\node[point,left of=localLip,node distance=20.5mm](point8){};
		\node[point,left of=inverse,node distance=17.5mm](point9){};
		\draw[-](compact)--(point6);
		\draw[-](point4)--(gammalow);
		\draw[->](point6)--(point4);
		\draw[-](localconver)--(point7);
		\draw[->](point7)--(point4);
		\draw[->](point7)--(point5);
		\draw[-](point5)--(gammasup);
		\draw[-](point8)--(localLip);
		\draw[->](point8)--(point5);
		\draw[-](point9)--(inverse);
		\draw[->](point9)--(point5);
		\draw[->](boundHolder)--(coercive);
		\node[point,right of=point,node distance=13mm](pointA){};
		\node[point,left of=gammaCov,node distance=13.5mm](pointB){};
		\node[point,left of=coercive,node distance=15.6mm](pointC){};
		\draw[-](pointA)--(point);
		\draw[->](pointB)--(pointA);
		\draw[->](pointC)--(pointA);
	\end{tikzpicture} 
	\vspace{2mm}
	\caption{Technical route for the  convergence analysis of $I^n$} 
	\label{F1}
\end{figure}
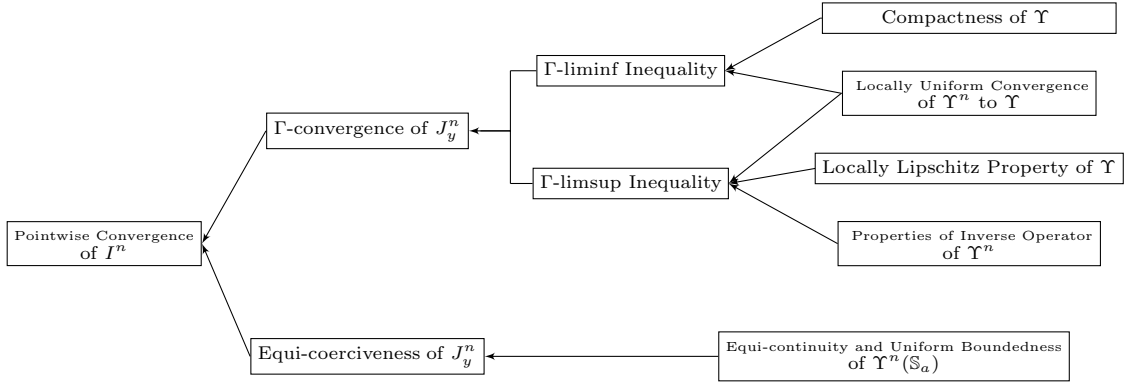

Hereafter, we always fix $\bar{x}\in\OO$ without extra statements.
For any $n\in\mbb N^+$, $y\in\mbb R$ and $f\in\mbf C(\OO_T;\mbb R)$, denote
\begin{align*}
	J_y(f):=\begin{cases}
		\inf\limits_{\{h\in\mbf L^2(\OO_T;\mbb R):~\Upsilon(h)=f\}}\frac{1}{2}\|h\|_{\mbf L^2(\OO_T)}^2,\quad &\text{if}~f(T,\bar x)=y,~f\in \text{Im}(\Upsilon),\\
		+\infty, & \text{otherwise},
	\end{cases}
\end{align*}
and 
\begin{align}\label{Jn}
	J^n_y(f):=\begin{cases}
		\inf\limits_{\{h\in\mbf L^2(\OO_T;\mbb R):~\Upsilon^n(h)=f\}}\frac{1}{2}\|h\|_{\mbf L^2(\OO_T)}^2,\quad &\text{if}~f(T,\bar x)=y,~f\in \text{Im}(\Upsilon^n),\\
		+\infty, & \text{otherwise}.
	\end{cases}
\end{align}
Then we rewrite rate functions $I^n$ and $I$ as, respectively,
\begin{align}\label{IandIn}
	I(y)=\inf_{\{f\in \mathbf C(\OO_T;\mbb R)\}}J_y(f),\quad\ I^n(y)=\inf_{\{f\in \mathbf C(\OO_T;\mbb R)\}}J^n_y(f),\quad y\in\mbb R.
\end{align}
As mentioned previously in the introduction,
we will give the pointwise convergence of $I^n$ by establishing the $\Gamma$-convergence and equi-coerciveness of $\{J_y^n\}_{n\in\mbb N^+}$, which relies on some qualitative properties of solution mappings of skeleton equations \eqref{eq:Skeleton} and \eqref{DisSkeleton}.

\subsection{Qualitative properties of $\Upsilon$ and $\Upsilon^n$}\label{Sec4.1}
Recall $\mbb S_a=\{h\in\mbf L^2(\OO_T;\mbb R):\|h\|_{\mbf L^2(\OO_T)}\le a\}$, $a\in[0,+\infty)$. 
The following proposition gives the uniform boundedness and H\"older continuity of $\Upsilon$  on bounded sets.
\begin{pro}\label{boundandHol}
	Let Assumptions \ref{assum1}-\ref{assum2} hold and $u_0\in\mbf C^2(\OO_T;\mbb R)$. Then we have
	\begin{itemize}
		\item [(1)] 	For any $a\ge 0$, $$\sup_{h\in \mbb S_a}\|\Upsilon(h)\|_{\mbf C(\OO_T)}\le K(a,T).$$
		\item[(2)] 	For any $a\ge0$, $\alpha\in(0,1)$, $s,t\in[0,T]$ and $x,y\in\OO$,
		$$\sup_{h\in\mbb S_a}|\Upsilon(h)(t,x)-\Upsilon(h)(s,y)|\le K(a,T,\alpha)(|x-y|+|t-s|^{\frac{3\alpha}{8}}).$$
	\end{itemize}
\end{pro}
\begin{proof}
	The proof of (1) is omitted since it
	is similar to that of \cite[Theorem 1.3]{Cardon2001}  but by replacing the stochastic integral
	by the deterministic integral containing $h$. Next we prove the second conclusion.
	
	For any $h\in\mbb S_a$, $a\in[0,+\infty)$, denote $f=\Upsilon(h)$. Then $f(t,x)=f_0(t,x)+f_1(t,x)+f_2(t,x)$, where
	$f_0(t,x):=\int_\OO G_t(x,z)u_0(z)\ud z$, $f_1(t,x):=\int_0^t\int_\OO\Delta G_{t-r}(x,z)b(f(r,z))\ud z\ud r$, $f_2(t,x):=\int_0^t\int_\OO G_{t-r}(x,z)\sigma(f(r,z))h(r,z)\ud z\ud r$.
	It follows from \cite[Lemma 2.3]{Cardon2001} that for any $(t,x),(s,y)\in\OO_T$,
	$|f_0(t,x)-f_0(s,y)|\le K(|t-s|^{\frac{1}{2}}+|x-y|)$.
	In addition, for any $(t,x),(s,y)\in\OO_T$ with $t\ge s$ and $\alpha\in(0,1)$, using the first conclusion of this lemma and \eqref{spaceHol1}-\eqref{timeHol1} gives
	\begin{align*}
		&\quad~|f_1(t,x)-f_1(s,y)|\le|f_1(t,x)-f_1(t,y)|+|f_1(t,y)-f_1(s,y)|\\
		&\le \int_{0}^t\int_\OO\big|(\Delta G_{t-r}(x,z)-\Delta G_{t-r}(y,z))b(f(r,z))\big|\ud z\ud r+\int_s^t\int_\OO|\Delta G_{t-r}(y,z)b(f(r,z))|\ud z\ud r\\
		&\quad+\int_{0}^s\int_\OO\big|(\Delta G_{t-r}(y,z)-\Delta G_{s-r}(y,z))b(f(r,z))\big|\ud z\ud r\\
		&\le K(a,T,\alpha)(|x-y|+|t-s|^{\frac{3\alpha}{8}}).
	\end{align*}
	Similarly, by the H\"older inequality, $h\in\mbb S_a$ and \eqref{spaceHol}-\eqref{timeHol}, it holds that for any $(t,x),(s,y)\in\OO_T$ and $\alpha\in(0,1)$,
	$$|f_2(t,x)-f_2(s,y)|\le K(a,T,\alpha)(|x-y|+|t-s|^{\frac{3\alpha}{8}}).$$
	Combining the above estimates finishes the proof. 
\end{proof}

The following two lemmas present the complete continuity and locally Lipschitz property of $\Upsilon$. 
\begin{lem}\label{UpsilonCompact}
	Let Assumptions \ref{assum1}-\ref{assum2} hold. Then	$\Upsilon$ is a completely continuous operator from $\mbf L^2(\OO_T;\mbb R)$ to $\mbf C(\OO_T;\mbb R)$, i.e., for every  sequence $\{h_n\}_{n\in\mbb N^+}$ weakly converging to $h$ in $\mbf L^2(\OO_T;\mbb R)$, the sequence  $\{\Upsilon (h_n)\}_{n\in\mbb N^+}$ converges to $\Upsilon (h)$ in the norm of $\mbf C(\OO_T;\mbb R)$.
\end{lem}
\begin{proof}
	It is shown in the proof of Theorem \ref{LDPofu} that 
	for any family $\{v^{\varepsilon}\}\subseteq \mscr A_M$ converging to some $v\in \mscr A_M$ in distribution as $\varepsilon\to 0$, $u^{0,v^{\varepsilon}}$ as $\mbf C(\OO_T;\mbb R)$-valued random variables converges to $u^{0,v}$  in distribution (see \eqref{controll} for the definition of $u^{\varepsilon,v}$). Especially, if a sequence   $\{h_n\}_{n\in\mbb N^+}$ of non-random functions weakly converges to $h$ in $\mbf L^2(\OO_T;\mbb R)$, then $\{u^{0,h_n}\}_{n\in \mbb N^+}$ converges to $u^{0,h}$ in $\mbf C(\OO_T;\mbb R)$. This immediately yields the conclusion by the fact $\Upsilon(g)=u^{0,g}$  for any $g\in\mbf L^2(\OO_T;\mbb R)$.  
\end{proof}

\begin{lem}\label{Upsilonh}
	Let Assumptions \ref{assum1} and \ref{assum2} hold.	Then for any $a\ge 0$ and $h_1,\,h_2\in\mbb S_a$, we have 
	$$\|\Upsilon(h_1)-\Upsilon(h_2)\|_{\mbf C(\OO_T)}\le K(a,T)\|h_1-h_2\|_{\mbf L^2(\OO_T)}.$$
\end{lem}

The proof of  Lemma \ref{Upsilonh} can be given by means of Proposition \ref{boundandHol} (1), Proposition \ref{Green} (1) and \eqref{fb-fa}, and  is  omitted. In what follows, we give some qualitative properties for the solution mapping $\Upsilon^n$ of the skeleton equation \eqref{DisSkeleton}. For this end, we present some useful properties of $G^n$ and $\Delta_n$.

\begin{pro}\label{disGreen}
	The following hold.
	\begin{itemize}
		\item [(1)] For any $\alpha\in(0,1)$,	there is $K(\alpha,T)>0$  such that for any  $(t,x),(s,y)\in\OO_T$ with $t\ge s$,
		\begin{gather*}
			\int_0^t\int_\OO|G^n_{t-r}(x,z)-G^n_{t-r}(y,z)|^2\ud z\ud r\le K(\alpha,T)|x-y|^2.\\	 \int_0^s\int_\OO|G^n_{t-r}(x,z)-G^n_{s-r}(x,z)|^2\ud z\ud r+\int_s^t\int_\OO|G^n_{t-r}(x,z)|^2\ud z\ud r\le K(\alpha,T)|t-s|^{\frac{3\alpha}{4}}.
		\end{gather*}
		\item[(2)]  For any $\alpha\in(0,1)$,	there is $K(\alpha,T)>0$  such that for any $(t,x),(s,y)\in\OO_T$ with $t\ge s$,
		\begin{gather*}
			\int_0^t\int_\OO|\Delta_{n,z} G^n_{t-r}(x,z)-\Delta_{n,z} G^n_{t-r}(y,z)|\ud z\ud r\le K(\alpha,T)|x-y|.\\	 \int_0^s\int_\OO|\Delta_{n,z} G^n_{t-r}(x,z)-\Delta _{n,z} G^n_{s-r}(x,z)|\ud z\ud r+\int_s^t\int_\OO|\Delta_{n,z} G^n_{t-r}(x,z)|\ud z\ud r\le K(\alpha,T)|t-s|^{\frac{3\alpha}{8}}.
		\end{gather*}
		\item [(3)] There is a constant $K(T)>0$ such that for any $(t,x)\in\OO_T$,
		\begin{gather*}
			\int_0^t\int_\OO|G^n_s(x,y)-G_s(x,y)|^2\ud y\ud s\le K(T)n^{-2}, \\
			\int_0^t\int_\OO|\Delta_{n,y}G^n_s(x,y)-\Delta G_s(x,y)|\ud y\ud s\le K(T)n^{-1}.
		\end{gather*}
	\end{itemize}
\end{pro}
The proofs of these estimates are similar to those of Lemmas 5.1-5.2 and Lemma 3.1 in \cite{Sheng}, and thus are omitted.

\begin{pro}\label{Laplacen}
	We have the following.
	\begin{itemize}
		\item [(1)] For any $t\in[0,T]$, $x,y\in\OO$, 
		\begin{gather*}
			\frac{\PD}{\PD t}G^n_t(x,y)=-\Delta_{n,y}^2G^n_t(x,y),\quad
			\frac{\PD}{\PD t}G^n_t(\kappa_n(x),y)=-\Delta_{n,x}^2G^n_t(\kappa_n(x),y).
		\end{gather*}
		\item [(2)] For any measurable function $w:\OO\to\mbb R$, $\Delta_n w(x_k)=(A_n\vec{w})_k$, $k=1,\ldots,n$, where $\vec{w}=(w(x_1),\ldots,w(x_n))^\top$.
		\item[(3)] For any measurable function $u,v:\OO\to\mbb R$, the discrete integration by parts formula holds:
		\begin{align*}
			\int_\OO \Delta_nu(x)v(\kappa_n(x))\ud x=\int_\OO u(\kappa_n(x))\Delta_nv(x)\ud x.
		\end{align*}
		\item [(4)] If $w\in\mbf C^2(\OO;\mbb R)$ with $w'(0)=w'(\pi)=0$, then $\|\Delta_n w\|_{\mbf C(\OO)}\le K\|w''\|_{\mbf C(\OO)}$. Further, if $w\in\mbf C^4(\OO;\mbb R)$ and $w(x)=w(0)$ for $x\in[0,x_4]$, and $w(x)=w(\pi)$ for $x\in[x_{n-3},\pi]$, then
		$\|\Delta_n^2 w\|_{\mbf C(\OO)}\le K\|w^{(4)}\|_{\mbf C(\OO)}$.
		\item[(5)] There is a constant $K>0$ such that for any Lipschitz continuous functions $u,v:\OO\to\mbb R$,
		\begin{align*}
			\|\Delta_{n}(uv)\|_{\mbf L^2(\OO)}\le K\big(\|u\|_{\mbf C(\OO)}\|\Delta_{n} v\|_{\mbf L^2(\OO)}+\|v\|_{\mbf C(\OO)}\|\Delta_{n} u\|_{\mbf L^2(\OO)}+[u]_{\mbf C^1(\OO)}[v]_{\mbf C^1(\OO)}\big).
		\end{align*}
	\end{itemize}
\end{pro}
\begin{proof}
	(1) This conclusion is a direct result of \eqref{Dnphi}.
	
	(2) The conclusion follows from the definitions of $A_n$ and
	$\Delta_n$.
	
	(3) Let $\vec{u}=(u(x_1),\ldots,u(x_n))^\top$ and $\vec{v}=(v(x_1),\ldots,v(x_n))^\top$. It follows from the second conclusion and the symmetry of $A_n$ that
	\begin{align*}
		\int_\OO \Delta_nu(x)v(\kappa_n(x))\ud x=\frac{\pi}{n}\sum_{k=1}^n\Delta_n u(x_k)v(x_k)=\frac{\pi}{n}\langle A_n\vec u,\vec v\rangle=\frac{\pi}{n}\langle \vec u,A_n\vec v\rangle=\int_\OO u(\kappa_n(x))\Delta_nv(x)\ud x.
	\end{align*}
	
	(4) Suppose that $w\in\mbf C^2(\OO;\mbb R)$ with $w'(0)=w'(\pi)=0$. For $x\in[\frac{(k-1)\pi}{n},\frac{k\pi}{n})$, $k=2,\ldots,n-1$, $\Delta_n w(x)=\frac{n^2}{\pi^2}\big(w(x_{k+1})-2w(x_k)+w(x_{k-1})\big)$. By the Taylor formula, it holds that $|\Delta_n w(x)|\le K\|w''\|_{\mbf C(\OO)}$.  For  $x\in[0,\frac{\pi}{n})$,
	$\Delta_n w(x)=\frac{n^2}{\pi^2}\big(w(x_2)-w(x_{1})\big)$. Using the Taylor formula and $w'(0)=0$ yields that for some $\eta\in(x_1,x_2)$,
	$$|\Delta_n w(x)|=\frac{n^2}{\pi^2}\Big|w'(x_1)\frac{\pi}{n}+\frac{1}{2}w''(\eta)\frac{\pi^2}{n^2}\Big|\le \frac{n}{\pi}|w'(x_1)-w'(0)|+\frac{1}{2}|w''(\eta)|\le K\|w''\|_{\mbf C(\OO)}.$$
	Similarly, one has $|\Delta_n w(x)|\le K\|w''\|_{\mbf C(\OO)}$ for $x\in[\frac{(n-1)\pi}{n},\pi]$.
	
	Next suppose that $w\in\mbf C^4(\OO;\mbb R)$ and $w(x)=w(0)$ for $x\in[0,x_4]$, and $w(x)=w(\pi)$ for $x\in[x_{n-3},\pi]$. A direct computation gives that for $x\in[\frac{(k-1)\pi}{n},\frac{k\pi}{n})$, $3\le k\le n-2$, 
	$$\Delta_n^2 w(x)=\frac{n^4}{\pi^4}\big(w(x_{k+2})-4w(x_{k+1})+6w(x_k)-4w(x_{k-1})+w(x_{k-2})\big).$$
	Again using the Taylor formula, we arrive at $|\Delta_n^2 w(x)|\le K\|w^{(4)}\|_{\mbf C(\OO)}$. It follows from $w(x)=w(0)$ for $x\in[0,x_4]$ and $w(x)=w(\pi)$ for $x\in[x_{n-3},\pi]$ that $\Delta_n w(x_i)=0$, for $i=1,2,3,n-2,n-1,n$. Thus, $\Delta_n^2 w(x_i)=0$ for $i=1,2,n-1,n$. In this way, $\|\Delta_{n}^2 w\|_{\mbf C(\OO)}\le K\|w^{(4)}\|_{\mbf C(\OO)}$.
	
	(5) By the definition of $\Delta_n$,
	\begin{align*}
		\Delta_n(uv)(x_1)&=\frac{n^2}{\pi^2}(u(x_2)v(x_2)-u(x_1)v(x_1))
		=\frac{n^2}{\pi^2}\big[(u(x_2)-u(x_1))v(x_2)+u(x_1)(v(x_2)-v(x_1))\big]\nonumber\\
		&=v(x_2)\Delta_{n}u(x_1)+u(x_1)\Delta_{n}v(x_1).
	\end{align*}
	Thus, \begin{align}\label{x1}
		|\Delta_n(uv)(x_1)|\le \|v\|_{\mbf C(\OO)}|\Delta_{n} u(x_1)|+ \|u\|_{\mbf C(\OO)}|\Delta_{n} v(x_1)|.
	\end{align}
	Similarly, one has
	\begin{align}\label{xn}
		|\Delta_n(uv)(x_n)|\le \|v\|_{\mbf C(\OO)}|\Delta_{n} u(x_n)|+ \|u\|_{\mbf C(\OO)}|\Delta_{n} v(x_n)|.
	\end{align}
	For $k=2,\ldots,n-1$, a direct computation gives
	\begin{align*}
		\Delta_{n} (uv)(x_k)&=\frac{n^2}{\pi^2}\big(u(x_{k+1})v(x_{k+1})-2u(x_k)v(x_k)+u(x_{k-1})v(x_{k-1})\big)\\
		&=u(x_{k-1})\Delta_{n} v(x_k)+v(x_k)\Delta_{n} u(x_k)+\frac{n^2}{\pi^2}(u(x_{k+1})-u(x_{k-1}))(v(x_{k+1})-v(x_k)).
	\end{align*}
	By the definition the semi-norm $[\cdot]_{\mbf C^1(\OO)}$, we obtain that for $k=2,\ldots,n-1$,
	\begin{align}\label{xk}
		|\Delta_{n} (uv)(x_k)|\le \|u\|_{\mbf C(\OO)}|\Delta_{n} v(x_k)|+\|v\|_{\mbf C(\OO)}|\Delta_{n} u(x_k)|+2[u]_{\mbf C^1(\OO)}[v]_{\mbf C^1(\OO)}.
	\end{align}
	Combining \eqref{x1}-\eqref{xn} yields
	\begin{align*}
		\|\Delta_{n}(uv)\|_{\mbf L^2(\OO)}^2&=\int_\OO|\Delta_{n}(uv)(x)|^2\ud x=\frac{\pi}{n}\sum_{k=1}^n|\Delta_n(uv)(x_k)|^2\\
		&\le K\|u\|_{\mbf C(\OO)}^2\frac{\pi}{n}\sum_{k=1}^n|\Delta_{n}v(x_k)|^2+K\|v\|_{\mbf C(\OO)}^2\frac{\pi}{n}\sum_{k=1}^n|\Delta_{n}u(x_k)|^2+K[u]^2_{\mbf C^1(\OO)}[v]^2_{\mbf C^1(\OO)}\\
		&= K\|u\|_{\mbf C(\OO)}^2\|\Delta_{n}v\|^2_{\mbf L^2(\OO)}+K\|v\|_{\mbf C(\OO)}^2\|\Delta_{n}u\|^2_{\mbf L^2(\OO)}+K[u]^2_{\mbf C^1(\OO)}[v]^2_{\mbf C^1(\OO)}.
	\end{align*}
	Thus, the last conclusion holds and the proof is complete. 
\end{proof}

Lemmas \ref{unbounded}-\ref{erroroder} establish the uniform boundedness and equi-continuity of $\Upsilon^n(\mbb S_a)$, and the locally uniform convergence of $\Upsilon^n$.   
\begin{lem}\label{unbounded}
	Let Assumptions \ref{assum1}-\ref{assum2} hold. Then	for any $a\ge0$, 
	\begin{align*}
		\sup_{h\in\mbb S_a}\sup_{n\ge1}\big(\|\Upsilon^n(h)\|_{\mbf C(\OO_T)}+\|\Delta_n \Upsilon^n(h)\|_{\mbf L^2(\OO_T)}\big)\le K(a,T).
	\end{align*}
\end{lem}
\begin{proof}
	In this proof, we will use some notations in the proof of Theorem \ref{LDPofun}.	
	Set $h\in\mbb S_a$, $a\ge0$ and $f=\Upsilon^n(h)$. 
	Let $\tilde{h}$ be  given as in \eqref{htilde}.
	Then 
	$\|\tilde{h}\|_{\mbf L^2(\OO_T)}\le \|h\|_{\mbf L^2(\OO_T)}\le a$ and $\Upsilon^n(\tilde{h})=f$, as is shown in the proof of Theorem \ref{LDPofun}. Define $\theta$ as in \eqref{theta}, and denote $\bm{h}=\theta^{-1}(h)$. Then we have $\Upsilon^n(\theta(\bm{h}))=f$, which is equivalent to
	$\mcal A(\bm{h})=\vec{f}$ with $\vec{f}(t)=(f(t,x_1),\cdots,f(t,x_n))$; see the proof of Theorem \ref{LDPofun}. Thus, one has
	\begin{align*}
		\vec{f}(t)=U^n(0)-\int_0^tA_n^2\vec f(s)\ud s+\int_0^tA_nB_n(\vec f(s))\ud s+\sqrt{\frac{n}{\pi}}\int_0^t\Sigma_n(\vec f(s))\bm{h}(s)\ud s,\quad t\in[0,T].
	\end{align*}
	By \eqref{thetanorm}, it holds that $\|\bm{h}\|_{\mbf L^2(0,T)}=\|\tilde{h}\|_{\mbf L^2(\OO_T)}\le a$. Thus from  Proposition \ref{Vbound} we conclude
	\begin{align}\label{eq10}
		\int_0^T\|A_n\vec{f}(t)\|_{l_n^2}^2\ud t+\sup_{t\in[0,T]}\|\vec{f}(t)\|_{l_n^\infty}\le K(a,T),
	\end{align}
	where  the definitions of the norms $\|\cdot\|_{l_n^2}$ and $\|\cdot\|_{l_n^\infty}$ are given in Appendix B.
	This implies that $\|f\|_{\mbf C(\OO_T)}=\sup_{t\in[0,T]}\|\vec{f}(t)\|_{l_n^\infty}\le K(a,T)$. In addition, according to Proposition \ref{Laplacen} (2) and \eqref{eq10},
	\begin{align*}
		\|\Delta_n f\|_{\mbf L^2(\OO_T)}^2=\int_0^T\int_\OO |\Delta_n f(t,x)|^2\ud x\ud t=\int_0^T\frac{\pi}{n}\sum_{k=1}^n|\Delta_n f(t,x_k)|^2\ud t=	\int_0^T\|A_n\vec{f}(t)\|_{l_n^2}^2\ud t\le K(a,T),
	\end{align*}
	which finishes the proof. 	
\end{proof}

\begin{lem}\label{UpsilonNholder}
	Let Assumptions \ref{assum1}-\ref{assum2} hold and $u_0\in\mbf C^2(\OO;\mbb R)$. Then for any $a\ge 0$ and $\alpha\in(0,1)$, there is $K(a,T,\alpha)>0$ such that for any $(t,x),(s,y)\in\OO_T$,
	\begin{align*}
		\sup_{h\in \mbb S_a}\sup_{n\ge1}|\Upsilon^n(h)(t,x)-\Upsilon^n(s,y)|\le K(a,T,\alpha)(|x-y|+|t-s|^{\frac{3\alpha}{8}}).
	\end{align*}	
\end{lem}
\begin{proof}
	Let $h\in\mbb S_a$, $a\ge0$ and denote $f:=\Upsilon^n(h)$. Then $f(t,x)=f_1(t,x)+f_2(t,x)+f_3(t,x)$, $(t,x)\in\OO_T$, where
	$f_1(t,x)=\int_\OO G^n_t(x,z)u_0(\kappa_n(z))\ud z$, $f_2(t,x)=\int_0^t\int_\OO \Delta_{n,z} G^n_{t-s}(x,z)b(f(s,\kappa_n(z)))\ud z\ud s$,\\ and $f_3(t,x)=\int_0^t\int_\OO G^n_{t-s}(x,z)\sigma(f(s,\kappa_n(z)))h(s,z)\ud z\ud s$.
	It follows from the definition of $G^n$ and the fact that $\sqrt{\frac{\pi}{n}}(\phi_j(x_1),\ldots,\phi_j(x_n))^\top$, $j=0,1,\ldots,n-1$ forms  the orthonormal basis of $\mbb R^n$ that 
	\begin{align}\label{Gn0}
		\int_\OO G^n_0(x,z)w(\kappa_n(z))\ud z&=\sum_{j=0}^{n-1}\phi_{j,n}(x)\sum_{k=1}^n\phi_j(x_k)w(x_k)\frac{\pi}{n}\nonumber\\&=\sum_{j=0}^{n-1}\sqrt{\frac{\pi}{n}}\Pi_n(\phi_j)(x)\sum_{k=1}^n\sqrt{\frac{\pi}{n}}\phi_j(x_k)w(x_k)=\Pi_n(w)(x)
	\end{align}	
	for any measurable function $w\in\OO\to\mbb R$.	
	By  (1) and (3) of Proposition \ref{Laplacen} and \eqref{Gn0},
	\begin{align*}
		f_1(t,x)&=\int_\OO\int_0^t\frac{\PD}{\PD s}G^n_s(x,z)u_0(\kappa_n(z))\ud s\ud z+\int_\OO G^n_0(x,z)u_0(\kappa_n(z))\ud z\\
		&=-\int_0^t\int_\OO \Delta_{n,z}G^n_s(x,z)\Delta_{n}u_0(z)\ud z\ud s+\Pi_n(u_0)(x).
	\end{align*}	
	Further, we have $\|\Delta_n u_0\|_{\mbf C(\OO)}\le K\|u_0''\|_{\mbf C(\OO)}$ due to Assumption \ref{assum2}, $u_0\in\mbf C^2(\OO;\mbb R)$ and Proposition \ref{Laplacen} (4). Thus, Proposition \ref{disGreen} (2) yields that for any $\alpha\in(0,1)$ and $(t,x),(s,y)\in\OO_T$ with $t\ge s$,
	\begin{align*}
		&\quad~|f_1(t,x)-f_1(s,y)|\le |f_1(t,x)-f_1(t,y)|+|f_1(t,y)-f_1(s,y)|\\
		&\le \int_0^t\int_\OO|\Delta_{n,z}G^n_s(x,z)-\Delta_{n,z}G^n_s(y,z)|\ud z\ud s+K\|u_0'\|_{\mbf C(\OO)}|x-y|+K\|u_0''\|_{\mbf C(\OO)}\int_s^t\int_\OO|\Delta_{n,z} G^n_r(y,z)|\ud z\ud s\\
		&\le K(|x-y|+|t-s|^{\frac{3\alpha}{8}}).
	\end{align*}
	Similarly to the proof of Proposition \ref{boundandHol} (2), one can show that for any $a\in(0,1)$ and $(t,x),(s,y)\in\OO_T$,
	\begin{align*}
		|f_i(t,x)-f_i(s,y)|\le K(a,T,\alpha)(|x-y|+|t-s|^{\frac{3\alpha}{8}}),~i=2,3.
	\end{align*}
	Thus the proof is complete.	 
\end{proof}

\begin{lem}\label{erroroder}
	Let Assumptions \ref{assum1}-\ref{assum2} hold and $u_0\in\mbf C^3(\OO;\mbb R)$. Then for any $a\ge 0$,
	$$\sup_{h\in\mbb S_a}\|\Upsilon^n(h)-\Upsilon(h)\|_{\mbf C(\OO_T)}\le K(a,T)n^{-1}.$$
\end{lem}
\begin{proof}
	Let $h\in\mbb S_a$, $a\ge0$ and denote $f:=\Upsilon(h)$, $f^n=\Upsilon^n(h)$. Then
	\begin{align*}
		f(t,x)-f^n(t,x)= T_1(t,x)+T_2(t,x)+T_3(t,x),~(t,x)\in\OO_T
	\end{align*}
	with
	\begin{gather*}
		T_1(t,x):=\int_\OO G_t(x,z)u_0(z)\ud z-\int_\OO G^n_t(x,z)u_0(\kappa_n(z))\ud z,\\
		T_2(t,x):=\int_0^t\int_\OO\Delta G_{t-s}(x,z)b(f(s,z))\ud z\ud s-\int_0^t\int_\OO\Delta_n G^n_{t-s}(x,z)b(f^n(s,\kappa_n(z)))\ud z\ud s,\\
		T_3(t,x):=\int_0^t\int_\OO G_{t-s}(x,z)\sigma(f(s,z))h(s,z)\ud z\ud s-\int_0^t\int_\OO G^n_{t-s}(x,z)\sigma(f^n(s,\kappa_n(z)))h(s,z)\ud z\ud s.
	\end{gather*}
	In the proof of Lemma \ref{UpsilonNholder}, we have shown that 
	\begin{align*}
		\int_\OO G^n_t(x,z)u_0(\kappa_n(z))\ud z=-\int_0^t\int_\OO \Delta_{n,z}G^n_s(x,z)\Delta_{n}u_0(z)\ud z\ud s+\Pi_n(u_0)(x).
	\end{align*}
	Similarly, using $\frac{\PD}{\PD t}G_t(x,y)=-\Delta^2G_t(x,y)$ and the Green formula gives
	\begin{align*}
		\int_\OO G_t(x,z)u_0(z)\ud z&=\int_\OO\int_0^t \frac{\PD}{\PD s}G_s(x,z)u_0(z)\ud z\ud s+\int_\OO G_0(x,z)u_0(z)\ud z\\
		&=-\int_0^t\int_\OO\Delta G_s(x,z)u_0''(z)\ud z\ud s+u_0(x).
	\end{align*}
	From the Taylor formula and the definition of $\Delta_{n}$, it can be verified that for any $z\in\OO$,
	\begin{align*}
		|\Delta_{n}u_0(z)-u_0''(z)|\le K\|u_0^{(3)}\|_{\mbf C(\OO)}n^{-1}.
	\end{align*}
	Further, combining the above formula, Proposition \ref{Green} (1),  Proposition \ref{disGreen} (3) and Proposition \ref{Laplacen} (4), we obtain
	\begin{align}\label{T1}
		|T_1(t,x)|\le&\; K\int_0^t\int_\OO \big|\Delta G_s(x,z)\big|\ud z\ud t\|u_0''-\Delta_{n} u_0\|_{\mbf C(\OO)}\nonumber\\
		&\;+\int_0^t\int_\OO|\Delta G_s(x,z)-\Delta_{n,z} G^n_s(x,z)\big|\ud z\ud s\|\Delta_{n} u_0\|_{\mbf C(\OO)}+|\Pi_n(u_0)(x)-u_0(x)|\nonumber\\
		\le&K(a,T)n^{-1}.
	\end{align}
	
	By \eqref{fb-fa}, Proposition \ref{boundandHol} and Lemma \ref{unbounded}, 
	\begin{align}\label{bf-bfn}
		&\quad|b(f(s,z))-b(f^n(s,\kappa_n(z)))|\nonumber\\
		&\le K\big(1+\|f\|_{\mbf C(\OO_T)}^2+\|f^n\|_{\mbf C(\OO_T)}^2\big)|f(s,z)-f(s,\kappa_n(z))+f(s,\kappa_n(z))-f^n(s,\kappa_n(z))|\nonumber\\
		&\le K(a,T)(n^{-1}+f(s,\kappa_n(z))-f^n(s,\kappa_n(z))).
	\end{align}
	It follows from \eqref{bf-bfn},  Proposition \ref{Green} (1), Proposition \ref{disGreen} (3) and the H\"older inequality that
	\begin{align}\label{T2}
		|T_2(t,x)|\le&\; \int_0^t\int_\OO \big|\Delta G_{t-s}(x,z)-\Delta_{n,z} G^n_{t-s}(x,z)\big|\ud z\ud s\|b(f^n)\|_{\mbf C(\OO_T)}\nonumber\\
		&\;+\int_0^t\int_\OO |\Delta G_{t-s}(s,z)||b(f(s,z))-b(f^n(s,\kappa_n(z)))|\ud z\ud s\nonumber\\
		\le&\; K(a,T)n^{-1}+K(a,T)\Big(\int_0^t(t-s)^{-\frac34}\|f^n(s,\cdot)-f(s,\cdot)\|^2_{\mbf C(\OO)}\ud s\Big)^{\frac12}\big(\int_0^t(t-s)^{-\frac34}\ud s\big)^{\frac12}\nonumber\\
		\le &\;K(a,T)n^{-1}+K(a,T)\Big(\int_0^t(t-s)^{-\frac34}\|f^n(s,\cdot)-f(s,\cdot)\|^2_{\mbf C(\OO)}\ud s\Big)^{\frac12}.
	\end{align}

	Similar to the proof of \eqref{bf-bfn}, one has
	\begin{align*}
		|\sigma(f(s,z))-\sigma(f^n(s,\kappa_n(z)))|\le K(a,T)(n^{-1}+f(s,\kappa_n(z))-f^n(s,\kappa_n(z))).
	\end{align*}
	Further, by  Proposition \ref{Green} (1), $h\in\mbb S_a$, Proposition \ref{disGreen} (3) and the H\"older inequality,
	\begin{align}\label{T3}
		|T_3(t,x)|^2\le
		K(a,T)n^{-2}+K(a,T)\int_0^t(t-s)^{-\frac12}\|f^n(s,\cdot)-f(s,\cdot)\|^2_{\mbf C(\OO)}\ud s.
	\end{align}
	Combining \eqref{T1}, \eqref{T2} and \eqref{T3}, we have
	\begin{align*}
		\|f^n(t,\cdot)-f(t,\cdot)\|^2_{\mbf C(\OO)}\le K(a,T)n^{-2}+K(a,T)\int_0^t(t-s)^{-\frac34}\|f^n(s,\cdot)-f(s,\cdot)\|^2_{\mbf C(\OO)}\ud s,
	\end{align*}
	which yields the conclusion by means of the Gronwall inequality. 
\end{proof}

Denote 
$\mcal M_n(\OO_T;\mbb R):=\Big\{w\in\mbf C(\OO_T;\mbb R):~w(0,\cdot)=\Pi_n (u_0)(\cdot),
w(t,\cdot)=\Pi_n(w(t,\cdot))~\text{for any}~ t\in[0,T],\text{and}~\frac{\PD}{\PD t}w(\cdot,\kappa_n(\cdot))\in\mbf L^2(\OO_T;\mbb R)\Big\}.$ Then we have the following result, which indicates that $\Upsilon^n$ is a bijection from $\mcal N_n(\OO_T;\mbb R)$ to $\mcal M_n(\OO_T;\mbb R)$, where $\mcal N_n(\OO_T;\mbb R)=\big\{h\in\mbf L^2(\OO_T;\mbb R):~h(t,x)=h(t,\kappa_n(x)),~(t,x)\in[0,T]\big\}$ is defined as in the proof of Theorem \ref{LDPofun}.
\begin{lem}\label{solveUpsilonN}Let Assumptions \ref{assum1}-\ref{assum2} hold, $u_0\in\mbf C^2(\OO;\mbb R)$. Then the following properties hold.\\
	(1) For any given $n\in\mbb N^+$, $\Upsilon^n(\mbf L^2(\OO_T;\mbb R))\subseteq \mcal M_n(\OO_T;\mbb R)$.\\
	(2) Let Assumption \ref{assum3} hold  and $f\in \mcal M_n(\OO_T;\mbb R)$. Then there is a unique $h\in\mcal N_n(\OO_T;\mbb R)$ such that $\Upsilon^n(h)=f$. And in this case, $h$ can be represented as
	\begin{align}\label{hExpress}
		h(t,x)=\frac{1}{\sigma(f(t,\kappa_n(x)))}\Big[\frac{\partial}{\partial t}f(t,\kappa_n(x))+\Delta_n^2 f(t,\kappa_n(x))-\Delta_{n}b(f(t,\kappa_n(x)))\Big],\, (t,x)\in\OO_T.
	\end{align}	
\end{lem}
\begin{proof}
	(1)	It suffices to show that $f:=\Upsilon^n(g)\in\mcal M_n(\OO_T;\mbb R)$ for any $g\in\mbf L^2(\OO_T;\mbb R)$ with $n\in\mbb N^+$ being fixed.  According to Lemma \ref{UpsilonNholder}, $f\in\mbf C(\OO_T;\mbb R)$. It is easy to see that $f(t,\cdot)=\Pi_n(f(t,\cdot))$. By \eqref{DisSkeleton} and \eqref{Gn0}, we have   $f(0,x)=\int_\OO G^n_0(x,y)u_0(\kappa_n(y))\ud y=\Pi_n(u_0)(x)$, $x\in \OO$.
	Further, it follows from \eqref{DisSkeleton} that for any $x\in\OO$, $f(t,x)$ is absolutely continuous w.r.t.\ $t$ and
	\begin{align}\label{PDtf}
		\quad~~&\;\frac{\PD}{\PD t}f(t,\kappa_n(x))\nonumber\\
		=	&\;\int_\OO \frac{\PD}{\PD t}G^n_t(\kappa_n(x),z)u_0(\kappa_n(z))\ud z+\int_0^t\int_\OO \frac{\PD}{\PD t}(\Delta_{n,z} G^n_{t-s}(\kappa_n(x),z))b(f(s,\kappa_n(z)))\ud z\ud s \nonumber\\
		&\;+\int_0^t\int_\OO \frac{\PD}{\PD t}G^n_{t-s}(\kappa_n(x),z)\sigma(f(s,\kappa_n(z)))g(s,z)\ud z\ud s\nonumber\\
		&\;+\int_\OO \Delta_{n,z}G^n_0(\kappa_n(x),z)b(f(t,\kappa_n(z)))\ud z+\int_\OO G^n_0(\kappa_n(x),z)\sigma(f(t,\kappa_n(z)))g(t,z)\ud z\nonumber\\
		= &\;-\Delta^2_{n,x}\Big[\int_\OO G^n_t(\kappa_n(x),z)u_0(\kappa_n(z))\ud z+\int_0^t\int_\OO \Delta_{n,z} G^n_{t-s}(\kappa_n(x),z)b(f(s,\kappa_n(z)))\ud z\ud s\nonumber\\
		&\;+\int_0^t\int_\OO G^n_{t-s}(\kappa_n(x),z)\sigma(f(s,\kappa_n(z)))g(s,z)\ud z\ud s\Big]+\Delta_{n}b(f(t,\kappa_n(x)))\nonumber\\
		&\;+\int_\OO G^n_0(\kappa_n(x),z)\sigma(f(t,\kappa_n(z)))g(t,z)\ud z\nonumber\\
		=&\;-\Delta_{n}^2f(t,\kappa_n(x))+\Delta_{n}b(f(t,\kappa_n(x)))+\int_\OO G^n_0(\kappa_n(x),z)\sigma(f(t,\kappa_n(z)))g(t,z)\ud z,
	\end{align} 
	where we have used Proposition \ref{Laplacen} (1) and (3), and \eqref{Gn0}.  Since $\|f\|_{\mbf C(\OO_T)}\le K(a,T)$ and $g\in\mbf L^2(\OO_T;\mbb R)$, we have $\|\frac{\PD}{\PD t}f(\cdot,\kappa_n(\cdot))\|_{\mbf L^\infty(\OO_T)}\le K(a,T,n)$.
	In this way, we get $f\in\mcal M_n(\OO_T;\mbb R)$.

	\vspace{1mm}
	(2)	We divide the proof of the second conclusion into two steps.
	
	\textbf{Step $1$: We prove that $h$ defined by \eqref{hExpress} satisfies $\Upsilon^n(h)=f$.} Denote 
	\begin{align*}
		\Phi(f,h)(t,x):
		=&\;\int_\OO G^n_t(x,z)u_0(\kappa_n(z))\ud z+\int_0^t\int_\OO \Delta_{n,z}G^n_{t-s}(x,z)b(f(s,\kappa_n(z)))\ud z\ud s\\
		&\;+\int_0^t\int_\OO G^n_{t-s}(x,z)\sigma(f(s,\kappa_n(z)))h(s,z)\ud z\ud s,\quad (t,x)\in\OO_T.
	\end{align*}
	Repeating the proof of the first conclusion,  one has $\Phi(f,h)\in \mcal M_n(\OO_T;\mbb R)$.
	Hence, it suffices to prove $\Phi(f,h)(t,\kappa_n(x))=f(t,\kappa_n(x))$, $(t,x)\in\OO_T$, due to $f\in\mcal M_n(\OO_T;\mbb R)$.
	Plugging \eqref{hExpress} into $\Phi(f,h)$ yields that for  $(t,x)\in\OO_T$,
	\begin{align*}
		\Phi(f,h)(t,\kappa_n(x))=&\;\int_\OO G^n_t(\kappa_n(x),z)u_0(\kappa_n(z))\ud z+\int_0^t\int_\OO G^n_{t-s}(\kappa_n(x),z)\frac{\PD}{\PD s}f(s,\kappa_n(z))\ud z\ud s\\
		&\;+\int_0^t\int_\OO G^n_{t-s}(\kappa_n(x),z)\Delta_{n,z}^2f(s,\kappa_n(z))\ud z\ud s,
	\end{align*}
	where we have used Proposition \ref{Laplacen} (3).
	By Proposition \ref{Laplacen} (1) and (3), 
	\eqref{Gn0} and $f\in\mcal M_n(\OO_T;\mbb R)$, we have
	\begin{align*}
		&\quad\Phi(f,h)(t,\kappa_n(x))\\
		&=\int_\OO G^n_t(\kappa_n(x),z)u_0(\kappa_n(z))\ud z+\int_{0}^t\int_\OO\frac{\PD}{\PD s}\big(G^n_{t-s}(\kappa_n(x),z)f(s,\kappa_n(z))\big)\ud z\ud s\\
		&=\int_\OO G^n_t(\kappa_n(x),z)u_0(\kappa_n(z))\ud z+\int_\OO G^n_0(\kappa_n(x),z)f(t,\kappa_n(z))\ud z-\int_\OO G^n_t(\kappa_n(x),z)u_0(\kappa_n(z))\ud z\\
		&=f(t,\kappa_n(x)),
	\end{align*}
	which implies $\Upsilon^n(h)=f$.

	\textbf{Step $2$: We prove that if $g\in\mcal N_n(\OO_T;\mbb R)$  satisfies $\Upsilon^n(g)=f$, then $g$ is given by the right-hand side of \eqref{hExpress}.} Let $\Upsilon^n(g)=f$ and $g\in\mcal N_n(\OO_T;\mbb R)$. It follows from  \eqref{PDtf}, \eqref{Gn0} and $g\in\mcal N_n(\OO_T;\mbb R)$ that 
	\begin{align*}
		\frac{\PD}{\PD t}f(t,\kappa_n(x))
		&=-\Delta_{n}^2f(t,\kappa_n(x))+\Delta_{n}b(f(t,\kappa_n(x)))+\int_\OO G^n_0(\kappa_n(x),z)\sigma(f(t,\kappa_n(z)))g(t,\kappa_n(z))\ud z\\
		&=-\Delta_{n}^2f(t,\kappa_n(x))+\Delta_{n}b(f(t,\kappa_n(x)))+\sigma(f(t,\kappa_n(x)))g(t,\kappa_n(x)).
	\end{align*}
	Thus, $g$ is given by the right-hand side of \eqref{hExpress} due to $\sigma\neq 0$, which finishes the proof. 
\end{proof}

\subsection{Pointwise convergence of $I^n$}
In this part, we use the results of Section \ref{Sec4.1} to prove the $\Gamma$-convergence and equi-coerciveness of $\{J^n_y\}_{n\in\mbb N^+}$, which further produces the pointwise convergence of $I^n$.

\begin{lem}\label{gammaLow}
	Suppose that Assumptions \ref{assum1}-\ref{assum2} hold and $u_0\in\mbf C^3(\OO;\mbb R)$.	Let   $y\in\mbb R$ be arbitrarily fixed.   Then for any $f\in\mbf C(\OO_T;\mbb R)$ and any sequence $\{f_n\}_{n\in\mbb N^+}$ converging to $f$ in $\mbf C(\OO_T;\mbb R)$, it holds that
	\begin{align}\label{gammaliminf}
		\liminf_{n\to+\infty}J^n_y(f_n)\ge J_y(f).
	\end{align}
\end{lem}
\begin{proof}
	This proof comes by repeating  the proof of \cite[Lemma 3.5]{LDPofSWE} word by word. 
\end{proof}

\begin{lem}\label{gammasup}
	Suppose that Assumptions \ref{assum1}-\ref{assum3} hold and $u_0\in\mbf C^3(\OO;\mbb R)$.	 Then for any subsequence $\{J^{n_k}_y\}_{k\in\mbb N^+}\subseteq \{J^{n}_y\}_{n\in\mbb N^+}$ with $y\in\mbb R$ being fixed, there is a subsubsequence $\{J^{n_{k_j}}_y\}_{j\in\mbb N^+}\subseteq \{J^{n_k}_y\}_{k\in\mbb N^+}$ such that
	\begin{itemize}
		\item[(1)] For any $f\in\mbf C(\OO_T;\mbb R)$, there exists a sequence $\{f_j\}_{j\in\mbb N^+}$ converging to $f$ in $\mbf C(\OO_T;\mbb R)$ and
		\begin{align}\label{fj}
			\limsup_{j\to+\infty} J^{n_{k_j}}_y(f_j)\le J_y(f).
		\end{align} 
		\item [(2)] For any $f\in\mbf C(\OO_T;\mbb R)$, 
		$\big(\Gamma\text{-}\limsup\limits_{j\to+\infty}J^{n_{k_j}}_y\big)(f)\le J_y(f).$
	\end{itemize}
\end{lem}
\begin{proof}
	(1) This proof can be given by following an almost same idea as shown in the proof of \cite[Lemma 3.8]{LDPofSWE}. We only present the main differences between this proof and that of \cite[Lemma 3.8]{LDPofSWE}.

	One only needs to prove \eqref{fj} for the case $J_y(f)<+\infty$. In this case, $f(T,\bar x)=y$ and $f\in\text{Im}(\Upsilon)$. Let $\{n_{k}\}_{k\in \mbb N^+}$ be any given subsequence. As is shown in the very beginning of the proof of \cite[Lemma 3.8]{LDPofSWE}, the following facts hold:
	(i) There is a  	subsequence $\{n_{k_j}\}_{j\in \mbb N^+}$ of $\{n_{k}\}_{k\in \mbb N^+}$ and $\tilde{h}_j\in\mcal N_{n_{k_j}}(\OO_T;\mbb R)$ such that 
	\begin{align}\label{k1}
		\limsup_{j\to+\infty}\|\tilde{h}_j\|^2_{\mbf L^2(\OO_T)}\le 2J_y(f).
	\end{align}
	(ii) The sequence $\tilde f_j:=\Upsilon^{n_{k_j}}(\tilde h_j)$, $j\in\mbb N^+$ converges to $f$ in $\mbf C(\OO_T)$-norm. 
	The latter implies that $\lim\limits_{j\to+\infty}y_j=y$ with $y_j:=\tilde{f}_j(T,\bar{x})$.
	
	\textbf{Case 1: $\bar{x}\in(0,\pi)$.}
	We first construct  ``modification function'' to modify $\tilde f_j$ to get our goal sequence $\{f_j\}_{j\in\mbb N^+}$ such that  $f_j(T,\bar{x})=y$, $j\in\mbb N^+$,
	$\{f_j\}_{j\in\mbb N^+}\subseteq \mcal M_{n_{k_j}}(\OO_T;\mbb R)$ and $\{f_j\}_{j\in\mbb N^+}$ converges to $f$ in $\mbf C(\OO_T;\mbb R)$. In what follows, we always let $j\in\mbb N^+$ be sufficiently large.  Choose $\eta_1,\eta_2\in\mbf C^\infty([0,1];\mbb R)$   such that $\eta_i^{(l)}(0)=\eta_i^{(l)}(1)=0$ for $i=1,2$ and $l=1,2\ldots$, and $\eta_1(0)=\eta_2(1)=0$, $\eta_1(1)=\eta_2(0)=1$. Then we define $p_j:\OO\to\mbb R$ by
	\begin{align*}
		p_j(x)=\begin{cases}
			0\qquad~&\text{if}~x\in[0,\frac{7\pi}{2n_{k_j}}],\\
			(y-y_j)\eta_1\Big(\frac{x-\frac{7\pi}{2n_{k_j}}}{\kappa_{n_{k_j}(\bar{x})}-\frac{\pi}{n_{k_j}}-\frac{7\pi}{2n_{k_j}}}\Big)\qquad~&\text{if}~x\in\big[\frac{7\pi}{2n_{k_j}},\kappa_{n_{k_j}}(\bar{x})-\frac{\pi}{n_{k_j}}\big],\\
			y-y_j\qquad~&\text{if}~x\in\big[\kappa_{n_{k_j}}(\bar{x})-\frac{\pi}{n_{k_j}},\kappa_{n_{k_j}}(\bar{x})+\frac{\pi}{n_{k_j}}\big],\\
			(y-y_j)\eta_2\Big(\frac{x-\kappa_{n_{k_j}}(\bar{x})-\frac{\pi}{n_{k_j}}}{\pi-\frac{7\pi}{2n_{k_j}}-\kappa_{n_{k_j}}(\bar{x})-\frac{\pi}{n_{k_j}}}\Big)\qquad~&\text{if}~x\in\big[\kappa_{n_{k_j}}(\bar{x})+\frac{\pi}{n_{k_j}},\pi-\frac{7\pi}{2n_{k_j}}\big],\\
			0	\qquad~&\text{if}~x\in\big[\pi-\frac{7\pi}{2n_{k_j}},\pi\big].
		\end{cases}
	\end{align*}
	Noting that $\lim\limits_{j\to+\infty}\kappa_{n_{k_j}}(\bar{x})=\bar{x}\in(0,\pi)$, we have that for any $l=0,1,\ldots$ and $x\in\big[\frac{7\pi}{2n_{k_j}},\kappa_{n_{k_j}}(\bar{x})-\frac{\pi}{n_{k_j}}\big]$,
	\begin{align*}
		\frac{\ud^l}{\ud x^l}\eta_1\Big(\frac{x-\frac{7\pi}{2n_{k_j}}}{\kappa_{n_{k_j}(\bar{x})}-\frac{\pi}{n_{k_j}}-\frac{7\pi}{2n_{k_j}}}\Big)=\Big({\kappa_{n_{k_j}(\bar{x})}-\frac{\pi}{n_{k_j}}-\frac{7\pi}{2n_{k_j}}}\Big)^{-l} \eta_1^{(l)}\Big(\frac{x-\frac{7\pi}{2n_{k_j}}}{\kappa_{n_{k_j}(\bar{x})}-\frac{\pi}{n_{k_j}}-\frac{7\pi}{2n_{k_j}}}\Big)\le K(\bar{x},l)\|\eta_1^{(l)}\|_{\mbf C([0,1])}.
	\end{align*}
	Similarly, one has that for any $l=0,1,\ldots$ and $x\in\big[\kappa_{n_{k_j}}(\bar{x})+\frac{\pi}{n_{k_j}},\pi-\frac{7\pi}{2n_{k_j}}\big]$,
	\begin{align*}
		\frac{\ud^l}{\ud x^l}	\eta_2\Big(\frac{x-\kappa_{n_{k_j}}(\bar{x})-\frac{\pi}{n_{k_j}}}{\pi-\frac{7\pi}{2n_{k_j}}-\kappa_{n_{k_j}}(\bar{x})-\frac{\pi}{n_{k_j}}}\Big)\le K(\bar{x},l)\|\eta_2^{(l)}\|_{\mbf C([0,1])}.
	\end{align*}
	Thus, $p_j\in\mbf  C^\infty(\OO;\mbb R)$ and for any $l=0,1,\ldots$,
	\begin{align}\label{pj}
		\|p_j^{(l)}\|_{\mbf C(\OO)}\le K(\bar{x},l)|y-y_j|.
	\end{align}
	
	Further,  we define  $f_j(t,x)=\tilde{f}_j(t,x)+\frac{t}{T}w_j(x)$, $(t,x)\in\OO_T$, where the modification term  
	$w_j:=\Pi_{n_{k_j}}(p_j)$. Thus, 
	it holds that $f_j(T,\bar{x})=y$ and $f_j\in\mcal M_{n_{k_j}}(\OO_T;\mbb R)$  due to $w_j(\bar{x})=p_j(\bar{x})=y-y_j$  and $\tilde f_j\in\mcal M_{n_{k_j}}(\OO_T;\mbb R)$, $j\in\mbb N^+$. Using the definition of $\Delta_{n_{k_j}}$, Proposition \ref{Laplacen} (4) and \eqref{pj} yields
	\begin{gather}
		\|w_j\|_{\mbf C(\OO)}\le 	\|p_j\|_{\mbf C(\OO)}\le K(\bar{x})|y-y_j|,\label{wj0}\\
		[w_j]_{\mbf C^1(\OO)}\le \|p_j'\|_{\mbf C(\OO)}\le K(\bar{x})|y-y_j|,\label{wj1}\\
		\|\Delta_{n_{k_j}}w_j\|_{\mbf C(\OO)}=\|\Delta_{n_{k_j}}\Pi_{n_{k_j}}(p_j)\|_{\mbf C(\OO)}=\|\Delta_{n_{k_j}}p_j\|_{\mbf C(\OO)}\le  K(\bar{x})|y-y_j|,\label{wj2}\\
		\|\Delta^2_{n_{k_j}}w_j\|_{\mbf C(\OO)}=\|\Delta^2_{n_{k_j}}\Pi_{n_{k_j}}(p_j)\|_{\mbf C(\OO)}=\|\Delta^2_{n_{k_j}}p_j\|_{\mbf C(\OO)}\le  K(\bar{x})|y-y_j|.\label{wj4}
	\end{gather}
	By \eqref{wj0}, $\lim\limits_{j\to+\infty}\|f_j-\tilde{f}_j\|_{\mbf C(\OO_T)}=0$ and hence $\lim\limits_{j\to+\infty}\|f_j-f\|_{\mbf C(\OO_T)}=0$.
	
	It follows from $\tilde{f}_j=\Upsilon^{n_{k_j}}(\tilde{h}_j)$, $\tilde{h}_j\in\mcal N_{n_{k_j}}(\OO_T;\mbb R)$, and Lemma \ref{solveUpsilonN} (2) that for $(t,x)\in\OO_T$ and $j\in\mbb N^+$,
	\begin{align}\label{hjtilde}
		\tilde{h}_j(t,x)=\frac{1}{\sigma(\tilde{f}_j(t,\kappa_{n_{k_j}}(x)))}\Big[\frac{\partial}{\partial t}\tilde{f}_j(t,\kappa_{n_{k_j}}(x))+\Delta^2_{n_{k_j}}\tilde{f}_j(t,\kappa_{n_{k_j}}(x))-\Delta_{n_{k_j}}b(\tilde{f}_j(t,\kappa_{n_{k_j}}(x)))\Big].
	\end{align}
	Further, for $(t,x)\in\OO_T$ and $j\in\mbb N^+$, define
	\begin{align}\label{hj}
		h_j(t,x):=\frac{1}{\sigma(f_j(t,\kappa_{n_{k_j}}(x)))}\Big[\frac{\partial}{\partial t}{f}_j(t,\kappa_{n_{k_j}}(x))+\Delta^2_{n_{k_j}}{f}_j(t,\kappa_{n_{k_j}}(x))-\Delta_{n_{k_j}}b(f_j(t,\kappa_{n_{k_j}}(x)))\Big].
	\end{align}
	Again by Lemma \ref{solveUpsilonN} (2), it follows that $h_j\in\mcal N_{n_{k_j}}(\OO_T;\mbb R)$ and $\Upsilon^{n_{k_j}}(h_j)=f_j$, $j\in\mbb N^+$. We claim
	\begin{align}\label{claim}
		\lim_{j\to+\infty}\|h_j-\tilde{h}_j\|_{\mbf L^2(\OO_T)}=0.
	\end{align}
	Using this claim,  $f_j\in\mcal D_{J^{n_{k_j}}_y}$ and \eqref{k1}, we derive
	\begin{align*}
		\limsup_{j\to+\infty} J^{n_{k_j}}_y(f_j)\le \frac{1}{2}\limsup_{j\to+\infty}\|h_j\|_{\mbf L^2(\OO_T)}^2\le \frac{1}{2}\limsup_{j\to+\infty}\|\tilde{h}_j\|_{\mbf L^2(\OO_T)}^2\le J_y(f),
	\end{align*}
	which proves \eqref{fj}. Thus,  \eqref{fj} is true once we justify the claim \eqref{claim}.
	
	Next, we prove \eqref{claim}. 
	Noting that
	\begin{align*}
		&\;\frac{\partial}{\partial t}f_j(t,\kappa_{n_{k_j}}(x))+\Delta^2_{n_{k_j}}f_j(t,\kappa_{n_{k_j}}(x))-\Big(\frac{\partial}{\partial t}\tilde f_j(t,\kappa_{n_{k_j}}(x))+\Delta^2_{n_{k_j}}\tilde f_j(t,\kappa_{n_{k_j}}(x))\Big)\\
		=&\;\frac{1}{T}w_j(\kappa_{n_{k_j}}(x))+\frac{t}{T}\Delta^2_{n_{k_j}}w_j(\kappa_{n_{k_j}}(x)), ~(t,x)\in\OO_T,
	\end{align*}
	we have $h_j(t,x)-\tilde{h}_j(t,x)=E_1(t,x)+E_2(t,x)+E_3(t,x)$, $(t,x)\in\OO_T$ with 
	\begin{gather*}
		E_1(t,x):=\Big(\frac{\partial}{\partial t}\tilde f_j(t,\kappa_{n_{k_j}}(x))+\Delta^2_{n_{k_j}}\tilde f_j(t,\kappa_{n_{k_j}}(x))\Big)\Big(\frac{1}{\sigma(f_j(t,\kappa_{n_{k_j}}(x)))}-\frac{1}{\sigma(\tilde f_j(t,\kappa_{n_{k_j}}(x)))}\Big),\\
		E_2(t,x):=\frac{w_j(\kappa_{n_{k_j}}(x))+t\Delta^2_{n_{k_j}}w_j(\kappa_{n_{k_j}}(x))}{T\sigma( f_j(t,\kappa_{n_{k_j}}(x)))},\\	E_3(t,x):=\frac{\Delta_{n_{k_j}}b(\tilde f_j(t,\kappa_{n_{k_j}}(x)))}{\sigma(\tilde f_j(t,\kappa_{n_{k_j}}(x)))}-\frac{\Delta_{n_{k_j}}b( f_j(t,\kappa_{n_{k_j}}(x)))}{\sigma( f_j(t,\kappa_{n_{k_j}}(x)))}.
	\end{gather*}
	Since $\lim\limits_{j\to+\infty}f_j=\lim\limits_{j\to+\infty}\tilde f_j=f$ in $\mbf C(\OO_T;\mbb R)$,  $\sup\limits_{j\in\mbb N^+}\big(\|f_j\|_{\mbf C(\OO_T)}+\|\tilde f_j\|_{\mbf C(\OO_T)}\big)<+\infty$, which yields 
	\begin{align}\label{inffj}
		\inf\limits_{j\in\mbb N^+}\inf\limits_{(t,x)\in\OO_T}|\sigma(f_j(t,x))|\ge c_1>0,\quad\inf\limits_{j\in\mbb N^+}\inf\limits_{(t,x)\in\OO_T}|\sigma(\tilde f_j(t,x))|\ge c_2>0
	\end{align}
	for two constants $c_1$ and $c_2$ due to Assumption \ref{assum3}. Thus, one immediately has $\lim\limits_{j\to\infty} E_2=0$ in $\mbf C(\OO_T;\mbb R)$ due to \eqref{wj0} and \eqref{wj4}. It follows from \eqref{inffj}, Assumption \ref{assum1} and $\lim\limits_{j\to\infty}\|f_j-\tilde f_j\|_{\mbf C(\OO_T)}=0$ that 
	\begin{align}\label{k2}
		\lim_{j\to+\infty}\sup_{(t,x)\in\OO_T}\Big|\frac{1}{\sigma(f_j(t,\kappa_{n_{k_j}}(x)))}-\frac{1}{\sigma(\tilde f_j(t,\kappa_{n_{k_j}}(x)))}\Big|=0.
	\end{align}
	It follows from $\tilde{f}_j=\Upsilon^{n_{k_j}}(\tilde h_j)$, \eqref{k1} and Lemmas \ref{unbounded} and  \ref{UpsilonNholder} that
	\begin{align}\label{k4}
		\sup_{j\in\mbb N^+}\big(\|\tilde{f}_j\|_{\mbf C(\OO_T)}+\|\Delta_{n_{k_j}}\tilde{f}_j\|_{\mbf L^2(\OO_T)}+\sup_{t\in[0,T]}[\tilde{f}_j(t,\cdot)]_{\mbf C^1(\OO)}\big)\le K(f,T).
	\end{align}
	By  Proposition \ref{Laplacen} (5) and \eqref{k4}, 
	\begin{align}\label{k6}
		\|\Delta_{n_{k_j}}\tilde{f}_j^2\|_{\mbf L^2(\OO_T)}\le K\|\tilde{f}_j\|_{\mbf C(\OO_T)}	\|\Delta_{n_{k_j}}\tilde{f}_j\|_{\mbf L^2(\OO_T)}+K\sup_{t\in[0,T]}[\tilde{f}_j(t,\cdot)]_{\mbf C^1(\OO)}^2\le K(f,T).
	\end{align}
	Similarly, one has $\|\Delta_{n_{k_j}}(\tilde{f}_j^3)\|_{\mbf L^2(\OO_T)}=\|\Delta_{n_{k_j}}(\tilde{f}_j^2\tilde{f}_j)\|_{\mbf L^2(\OO_T)}\le K(f,T)$, which  implies
	\begin{align}\label{Deltabf}
		\|\Delta_{n_{k_j}}b(\tilde{f}_{j})\|_{\mbf L^2(\OO_T)}\le K(f,T).
	\end{align}
	This combined with \eqref{k1}, \eqref{hjtilde}  and the boundedness of $\sigma$ implies
	\begin{align}\label{k3}
		\sup_{j\in\mbb N^+}\int_0^T\int_\OO \Big|\frac{\partial}{\partial t}\tilde f_j(t,\kappa_{n_{k_j}}(x))+\Delta^2_{n_{k_j}}\tilde f_j(t,\kappa_{n_{k_j}}(x))\Big|^2\ud x\ud t<+\infty.
	\end{align}
	Applying \eqref{k2} and \eqref{k3}, we arrive at
	$\lim\limits_{j\to\infty} E_1=0$ in $\mbf L^2(\OO_T;\mbb R)$.
	
	As for $E_3$, we have
	\begin{align*}
		\Delta_{n_{k_j}}b({f}_{j})-\Delta_{n_{k_j}}b(\tilde f_{j})=
		\Delta_{n_{k_j}}\big[(f_{j}-\tilde{f}_{j})(f_{j}^2+f_{j}\tilde{f}_{j}+\tilde{f}^2_{j}-1)\big]=\frac{t}{T}\Delta_{n_{k_j}}(w_jV_j),
	\end{align*}
	with $V_j:=f_{j}^2+f_{j}\tilde{f}_{j}+\tilde{f}^2_{j}-1=3\tilde{f}_j^2+\frac{3t}{T}\tilde{f}_{j}w_j+\frac{t^2}{T^2}w_j^2-1$.
	Using \eqref{wj0}-\eqref{wj2} and Proposition \ref{Laplacen} (5)  yields
	\begin{align*}
		\|\Delta_{n_{k_j}}w_j^2\|_{\mbf L^2(\OO)}\le K\|w_j\|_{\mbf C(\OO)}	\|\Delta_{n_{k_j}}w_j\|_{\mbf L^2(\OO)}+K[w_j]_{\mbf C^1(\OO)}^2\le K(\bar{x})|y-y_j|^2.
	\end{align*}
	Similarly, one has
	\begin{align*}
		&\;\|\Delta_{n_{k_j}}(\tilde{f}_jw_j)\|_{\mbf L^2(\OO_T)}\\
		\le&\; K(T)\big(\|\tilde{f}_j\|_{\mbf C(\OO_T)}	\|\Delta_{n_{k_j}}w_j\|_{\mbf L^2(\OO)}+\|w_j\|_{\mbf C(\OO)}	\|\Delta_{n_{k_j}}\tilde{f}_j\|_{\mbf L^2(\OO_T)}+\sup_{t\in[0,T]}[\tilde{f}_j(t,\cdot)]_{\mbf C^1(\OO)}[w_j]_{\mbf C^1(\OO)}\big)\\
		\le &\;K(\bar{x},f,T)|y-y_j|.
	\end{align*}
	The above estimates and \eqref{k6} indicate that 
	\begin{align}\label{vj}
		\sup_{j\in\mbb N^+}\big(\|V_j\|_{\mbf C(\OO_T)}+\|\Delta_{n_{k_j}}V_j\|_{\mbf L^2(\OO_T)}\big)\le K(\bar{x},f,T)(1+|y-y_j|+|y-y_j|^2).
	\end{align}
	Notice that 
	\begin{align}\label{VjC1}
		\sup_{t\in[0,T]}[V_j(t,\cdot)]_{\mbf C^1(\OO)}\le&\; K \big(\sup_{t\in[0,T]}[\tilde{f}_j(t,\cdot)]_{\mbf C^1(\OO)}+[w_j]_{\mbf C^1(\OO)}\big)\big(\|\tilde{f}_j\|_{\mbf C(\OO_T)}+\|w_j\|_{\mbf C(\OO)}\big)\nonumber\\
		\le&\; K(\bar{x},f,T)(1+|y-y_j|+|y-y_j|^2).
	\end{align}
	Combining \eqref{wj0}-\eqref{wj2} and \eqref{vj}-\eqref{VjC1}, we use Proposition \ref{Laplacen} (5) to get
	\begin{align*}
		\|\Delta_{n_{k_j}}b(\tilde{f}_j)-\Delta_{n_{k_j}}b({f}_j)\|_{\mbf L^2(\OO_T)}\le K(\bar{x},f,T)|y-y_j|(1+|y-y_j|+|y-y_j|^2),
	\end{align*}
	which gives
	\begin{align}\label{k5}
		\lim_{j\to+\infty}\|\Delta_{n_{k_j}}b(\tilde{f}_j)-\Delta_{n_{k_j}}b({f}_j)\|_{\mbf L^2(\OO_T)}=0.
	\end{align} 
	By \eqref{inffj}-\eqref{k2},  \eqref{Deltabf} and \eqref{k5}, we derive $E_3\to 0$ in $\mbf L^2(\OO_T;\mbb R)$ as $j\to+\infty$. 
	In this way, we prove the claim \eqref{claim} and obtain \eqref{fj}. 
	
	\textbf{Case 2: $\bar{x}=0,\pi$.} In this case, \eqref{fj} can be proved similarly as in Case 1. 
	One only needs to modify the expression of $p_j$  accordingly such that $p_j(T,0)=y-y_j$ or $p_j(T,\pi)=y-y_j$. Then using $w_j=\Pi_{n_{k_j}}(p_j)$ as the modification term to complete the proof.

	(2) Taking $\{f_j\}_{\in\mbb N^+}$ be the sequence satisfying \eqref{fj}, we obtain that $\big(\Gamma$-$\limsup\limits_{j\to+\infty}J_y^{n_{k_j}}\big)(f):=\inf\big\{\limsup\limits_{j\to+\infty}J^{n_{k_j}}_y(F_j):\lim\limits_{j\to+\infty}F_j=f~\text{in}~\mbf C(\OO_T;\mbb R)\big\}\le \limsup\limits_{j\to+\infty}J_y^{n_{k_j}}(f_j)\le J_y(f)$.  
\end{proof}

\begin{theo}\label{JyGamma}
	Suppose that Assumptions \ref{assum1}-\ref{assum3} hold and $u_0\in\mbf C^3(\OO;\mbb R)$. Let $y\in\mbb R$ be fixed. Then for any subsequence $\{J^{n_k}_y\}_{k\in\mbb N^+}$ of $\{J^n_y\}_{n\in\mbb N^+}$, there is a subsubsequence $\{J^{n_{k_j}}_y\}_{j\in\mbb N^+}$ which $\Gamma$-converges to $J_y$ on $\mbf C(\OO_T;\mbb R)$. Thus, $\{J^n_y\}_{n\in\mbb N^+}$ $\Gamma$-converges to $J_y$ on $\mbf C(\OO_T;\mbb R)$.
\end{theo}
\begin{proof}
	It follows from Lemma \ref{gammasup} (2), for any subsequence $\{J^{n_k}_y\}_{k\in\mbb N^+}$ of $\{J^n_y\}_{n\in\mbb N^+}$, there is a subsubsequence $\{J^{n_{k_j}}_y\}_{j\in\mbb N^+}$ such that
	$\big(\Gamma$-$\limsup\limits_{j\to+\infty}\ J^{n_{k_j}}_y\big)(f)\le J_y(f)$ for any $f\in\mbf C(\OO_T;\mbb R)$. 
	Further,	by Lemma \ref{gammaLow}, we have that for any $f\in\mbf C (\OO_T;\mbb R)$,
	\begin{align*}
		\big(\Gamma\text{-}\liminf_{n\to+\infty}J^n_y\big)(f):=\inf\big\{\liminf_{n\to+\infty}J^n_y(f_n):~\lim_{n\to+\infty}f_n=f~\text{in}~\mbf C(\OO_T;\mbb R)\big\}\ge J_y(f),
	\end{align*}
	which along with Proposition \ref{subgamma1} implies $\big(\Gamma$-$\liminf\limits_{j\to+\infty} J^{n_{k_j}}_y\big)(f)\ge J_y(f)$. Thus, it holds that  $\{J^{n_{k_j}}_y\}_{j\in\mbb N^+}$ $\Gamma$-converges to $J_y$ on $\mbf C(\OO_T;\mbb R)$. Finally, the proof is complete according to Proposition \ref{subgamma2}. 
\end{proof}

\begin{lem}\label{equi}
	Let Assumptions \ref{assum1}-\ref{assum2} hold and $u_0\in\mbf C^2(\OO;\mbb R)$.	For any $y\in\mbb R$, $\{J^n_y\}_{n\in\mbb N^+}$ is equi-coercive on $\mbf C(\OO_T;\mbb R)$.
\end{lem}
\begin{proof}
	Using Lemmas \ref{unbounded} and \ref{UpsilonNholder}, one can complete the proof by following  the same idea as in the proof of \cite[Lemma 3.10]{LDPofSWE}. 
\end{proof}

\begin{theo}\label{pointconvergence}
	Let  Assumption \ref{assum1}-\ref{assum3} hold and $u_0\in\mbf C^3(\OO;\mbb R)$.  Then 
	$		\lim\limits_{n\to+\infty}I^n(y)=I(y),~ y\in\mbb R.$	
\end{theo}
\begin{proof}
	For any given $y\in\mbb R$, it follows from  Lemma \ref{equi} and Theorem \ref{JyGamma} that $\{J_y^n\}_{n\in\mbb N^+}$ is equi-coercive and $\Gamma$-converges to $J_y$ on $\mbf C(\OO_T;\mbb R)$. Thus Theorem \ref{gammatheorem} and \eqref{IandIn} finish the proof.  
\end{proof}

\appendix
\section*{Appendix}
\setcounter{equation}{0}
\setcounter{subsection}{0}
\setcounter{Def}{0}
\renewcommand{\theDef}{A.\arabic{Def}}
\renewcommand{\theequation}{A.\arabic{equation}}
\renewcommand{\thesubsection}{A.\arabic{subsection}}

\subsection*{A: $\Gamma$-convergence}\label{AppenA}
In this part, we introduce some definitions and results in the theory of $\Gamma$-convergence.  We refer the interested readers to \cite{Gamma93,Gamma18} for more details on $\Gamma$-convergence.  Let $X$ be a metric space and $\overline{\mbb R}=\mbb R\cup\{\pm\infty\}$ denote the set of extended real numbers. In this part, we always let $F_n,\,F:X\to\overline{\mbb R}$, $n\in\mbb N^+$ be given functionals.

\begin{Def}\label{Gammadef1}
	The sequence $\{F_n\}_{n\in\mbb N^+}$ is said to $\Gamma$-converge to $F$, if
	\begin{itemize}
		\item[(1)] For all sequences $\{x_n\}_{n\in\mbb N^+} \subseteq X$ with $\lim\limits_{n\to+\infty}x_n=x$ in $X$, the  \textbf{liminf inequality} holds:
		$$\liminf_{n\to+\infty}F_n(x_n)\ge F(x).$$
		\vspace{-5mm}
		\item[(2)] For any $x\in X$, there is a \textbf{recovery sequence} $\{x_n\}_{n\in\mbb N^+}$ such that $\lim\limits_{n\to+\infty}x_n=x$ in $X$ and 
		$$\limsup_{n\to+\infty}F_n(x_n)\le F(x).$$	
	\end{itemize}
\end{Def}
\begin{rem}
	Notice that under the first condition of Definition \ref{Gammadef1}, the second condition is equivalent to that for any $x\in X$, there is $\{x_n\}_{n\in\mbb N^+}$ converging to $x$ in $X$ and 
	$\lim\limits_{n\to+\infty}F_n(x_n)= F(x).$
\end{rem}

\begin{Def}\label{Gammadef2}
	The \textbf{$\Gamma$-lower limit}  and the
	\textbf{$\Gamma$-super limit} of $\{F_n\}_{n\in\mbb N^+}$ are, respectively,
	\begin{gather*}
		(\Gamma\text{-}\liminf_{n\to+\infty}F_n)(x)=\inf\Big\{\liminf_{n\to+\infty}F_n(x_n):~x_n\to x~\text{in}~X\Big\},\quad x\in X,\\
		(\Gamma\text{-}\limsup_{n\to+\infty}F_n)(x)=\inf\Big\{\limsup_{n\to+\infty}F_n(x_n):~x_n\to x~\text{in}~X\Big\},\quad x\in X.
	\end{gather*}
	If  $\Gamma\text{-}\liminf\limits_{n\to+\infty}F_n=\Gamma\text{-}\limsup\limits_{n\to+\infty}F_n=F$, then we write $F=\Gamma\text{-}\lim\limits_{n\to+\infty} F_n$ and we say that the sequence $\{F_n\}_{n\in\mbb N^+}$ $\Gamma$-converges to $F$ (on $X$) or that $F$ is the $\Gamma$-limit of $\{F_n\}_{n\in\mbb N^+}$ (on $X$).
\end{Def}
Readers can refer to \cite[Definition 4.1, Proposition 8.1]{Gamma93} and \cite[Section 13.1]{Gamma18} on the equivalence of Definitions \ref{Gammadef1} and \ref{Gammadef2}.

The following give some relationships between the $\Gamma$-limit of a sequence of functionals and the $\Gamma$-limit of its subsequence.
\begin{pro}\textup{\cite[Proposition 6.1]{Gamma93}}\label{subgamma1}
	If $\{F_{n_k}\}_{k\in\mbb N^+}$ is a subsequence of $\{F_n\}_{n\in\mbb N^+}$, then 
	\begin{gather*}
		\Gamma\text{-}\liminf_{n\to+\infty}F_n\le \Gamma\text{-}\liminf_{k\to+\infty}F_{n_k},\quad \Gamma\text{-}\limsup_{n\to+\infty}F_n\ge \Gamma\text{-}\limsup_{k\to+\infty}F_{n_k}.
	\end{gather*}
\end{pro}
\begin{pro}\textup{\cite[Proposition 8.3]{Gamma93}}\label{subgamma2}
	$\{F_n\}_{n\in\mbb N^+}$ $\Gamma$-converges to $F$ on $X$ if and only if every subsequence
	of $\{F_n\}_{n\in\mbb N^+}$ contains a further subsequence which $\Gamma$-converges to $F$.	
\end{pro}

Next, we introduce the  well-known result, concerning the variational calculus,  in the theory of $\Gamma$-convergence.
\begin{Def}\textup{\cite[Definition 7.6]{Gamma93}}\label{equidef}
	We say that the sequence $\{F_n\}_{n\in\mbb N^+}$ is equi-coercive (on $X$), if
	for every $t\in\mbb R$, there exists a compact subset $K_t$ of $X$ such
	that $\{F_n\le t\}\subseteq K_t$ for every $n\in\mbb N^+$.	
\end{Def}
\begin{theo}\textup{\cite[Theorem 7.8]{Gamma93}}\label{gammatheorem}
	If $\{F_n\}_{n\in\mbb N^+}$ is equi-coercive and $\Gamma$-converges to $F$ on $X$, then 
	$\min\limits_{x\in X}F(x)=\lim\limits_{n\to+\infty}\inf\limits_{x\in X}F_n(x).$
\end{theo}

\setcounter{equation}{0}
\setcounter{subsection}{0}
\setcounter{Def}{0}
\renewcommand{\theDef}{B.\arabic{Def}}
\renewcommand{\theequation}{B.\arabic{equation}}
\renewcommand{\thesubsection}{B.\arabic{subsection}}
\subsection*{B: Regularity estimate for an ordinary differential equation}\label{AppenB}
In this part, we provide a regularity estimate for an ordinary differential equation which is equivalent to  the skeleton equation \eqref{DisSkeleton} to some sense. 

We introduce the discrete $L^2$-inner product and the discrete $L^p$-norm ($1\le p\le\infty$), respectively, as
\begin{align*}
	\langle a,b\rangle_{l^2_n}=\frac{\pi}{n}\sum_{i=1}^{n}a_ib_i,\qquad\|a\|_{l^p_n}=\begin{cases}\left(\frac{\pi}{n}\sum\limits_{i=1}^{n}|a_i|^p\right)^{\frac{1}{p}},\quad&1\le p<\infty,\\
		\sup\limits_{1\le j\le n}|a_i|,\quad&p=\infty\end{cases}
\end{align*}
for the vectors $a=(a_1,\ldots,a_{n})^\top$ and $b=(b_1,\ldots,b_{n})^\top$. By the H\"older inequality, one immediate has that for  $p\ge 1$ and $q\in[p,+\infty]$, there is $K(p,q)>0$ independent of $n$ such that for any $a\in\mbb R^n$,
\begin{align}\label{lnnorm}
	\|a\|_{l_n^p}\le K(p,q)\|a\|_{l_n^q}.
\end{align}
The following estimates, viewed as discrete interpolation inequalities, can be proved similar to \cite[Lemma 3.1]{Sheng1}.
\begin{pro}\label{Disinterpo}
	Let $2\le p\le\infty$ and $t>0$. Then it holds for all $a=(a_1,\ldots,a_n)\in l_n^\infty$, $n\ge1$ that
	\begin{gather}\label{l2H1}
		\|a\|_{l^\infty_n}\le\sqrt{\pi}\|a\|_{l_n^2}+\|(- A_n)^{\frac{1}{2}}a\|_{l^2_n},\\\label{l6h2}
		\|a\|_{l_n^6}\le C\| A_na\|_{l^2_n}^{\frac{1}{6}}\|a\|^{\frac{5}{6}}_{l_n^2}+C\|a\|_{l_n^2},\\\label{interpolation}
		\|e^{- A_n^2t}a\|_{l_n^p}\le (1+Ct^{-\frac{1}{4}(\frac{1}{2}-\frac{1}{p})})\|a\|_{l_n^2},
	\end{gather}
	where $C>0$ is a constant independent of $a$, $n$ and $t>0$.
\end{pro}

Given $\nu\in\R$, define $(-\dot A_n)^{\nu}:\R^n\rightarrow\R^n$ by $$(-\dot A_n)^\nu a
=\sum_{k=1}^{n-1} (-\lambda_{k,n})^{\nu}\langle a,e_k\rangle  e_k,\quad\forall~a\in\R^n,$$
which coincides with $(- A_n)^\nu a=\sum_{k=0}^{n-1} (-\lambda_{k,n})^{\nu}\langle a,e_k\rangle  e_k$ when $\nu>0$.
Besides, it is readily to verify that for $\mu\ge 0$ and $\nu\in\mbb R$,
\begin{align}\label{eq1}
	\langle(- A_n)^{\mu} a,(-\dot A_n)^{\nu}b\rangle  =\langle (-\dot A_n)^{\mu+\nu}a,b\rangle  ,\quad\forall~ a,b\in\R^n.
\end{align}
We also denote $\mbb L a:=\frac{1}{\sqrt{n}}\langle a,e_0\rangle=\frac{1}{n}\sum_{k=1}^n{a_k}$ for $a=(a_1,\ldots,a_n)\in\R^n$. 
Then it holds that for any $a\in\mbb R^n$,
\begin{align}\label{mbbL}
	\sqrt{n}|\mbb L a|\le |a|.
\end{align}
In addition, by the \eqref{eq1}, we have
\begin{align}\label{LaLb}
	\langle (-\dot A_n)^{-1}a, A_nb\rangle=-\langle a,b\rangle+n\mbb L a\cdot\mbb L b,\quad\forall~a,b\in\R^n.
\end{align}

\begin{pro}\label{Vbound}
	
	Let $V\in\mbf C([0,T];\mbb R^n)$ be the solution to the following differential equation
	\begin{align}\label{SDE}
		\begin{cases}
			\dot{V}(t)=-A_n^2V(t)+A_nB_n(V(t))+\sqrt{\frac{n}{\pi}}\Sigma_n(V(t))R(t), ~t\in(0,T],\\
			V(0)=U^n(0),
		\end{cases}
	\end{align} 
	with $R\in\mbf L^2(0,T;\mbb R^n)$. Let Assumptions \ref{assum1}-\ref{assum2} be satisfied. Then for any $a\ge 0$ and $R\in\mbf L^2(0,T;\mbb R^n)$ with $\|R\|_{\mbf L^2(0,T)}\le a$, there is $K(a,T)>0$ independent of $n$ such that
	\begin{align*}
		\int_0^T\|A_nV(t)\|_{l_n^2}^2\ud t+\sup_{t\in[0,T]}\|V(t)\|_{l_n^\infty}\le K(a,T).
	\end{align*}
\end{pro}
\begin{proof}
	Fix $R\in\mbf L^2(0,T;\mbb R^n)$ with $\|R\|_{\mbf L^2(0,T)}\le a$.	We divide the proof into four steps.
	
	\textbf{Step $1$: We show that
		\begin{align}\label{VH1}
			\int_{0}^T\|(-A_n)^{\frac{1}{2}}V(t)\|_{l_n^2}^2\ud t+\int_0^T\|V(t)\|_{l_n^4}^4\ud t\le K(a,T).
		\end{align}
	}

	Noting that $\mbb LA_na=0$, for any $a\in\mbb R^n$, we deduce from \eqref{SDE} that
	\begin{align*}
		\mbb L\dot{V}(t)=\sqrt{\frac{n}{\pi}}\mbb L(\Sigma_n(V(t))R(t)).
	\end{align*}
	Then using \eqref{mbbL}, $u_0\in\mbf C(\OO;\mbb R)$  and the boundedness of $\sigma$ yields 
	\begin{align}\label{LV}
		\mbb LV(t)&=\mbb LU^n(0)+\sqrt{\frac{n}{\pi}}\int_0^t\mbb L(\Sigma_n(V(s))R(s))\ud s\le \|u_0\|_{\mbf C(\OO)}+K\int_0^t|R(s)|\ud s\nonumber\\
		&\le  \|u_0\|_{\mbf C(\OO)}+K(T)\|R\|_{\mbf L^2(0,T)}\le K(a,T).
	\end{align}
	It follows from $|b(x)|\le K(1+|x|^3)$, $\LL \cdot,\cdot\RR_{l_n^2}=\frac{\pi}{n}\LL\cdot,\cdot\RR$, \eqref{lnnorm}, \eqref{LaLb}, \eqref{LV} and the Young inequality that for any $\eta\in(0,1)$,
	\begin{align}\label{eq3}
		&\quad\LL A_nB_n(V(t)),(-\dot{A}_n)^{-1}V(t)\RR_{l_n^2}=-\frac{\pi}{n}\LL B_n(V(t)),V(t)\RR+\pi\mbb LB_n(V(t))\cdot\mbb LV(t)\nonumber\\
		&\le -\frac{\pi}{n}\sum_{j=1}^n(V_j^4(t)-V_j^2(t))+\frac{K(a,T)}{n}\sum_{j=1}^n(1+|V_j(t)|^3)\nonumber\\
		&\le - \|V(t)\|_{l_n^4}^4+K(a,T)(1+\|V(t)\|^2_{l_n^4}+\|V(t)\|^3_{l_n^4})\nonumber\\
		&\le-(1-\eta)\|V(t)\|_{l_n^4}^4+K(\eta,a,T).
	\end{align} 
	By the boundedness of $\sigma$ and  $\LL \cdot,\cdot\RR_{l_n^2}=\frac{\pi}{n}\LL\cdot,\cdot\RR$, \begin{align}\label{Rt}
		\sqrt{\frac{n}{\pi}}\|\Sigma_n(V(t))R(t)\|_{\l_n^2}\le K\sqrt{\frac{n}{\pi}}\|R(t)\|_{\l_n^2}\le K|R(t)|,
	\end{align}
	which together with $\|(-\dot{A}_n)^{-1}\cdot\|_{l_n^2}\le K\|(-\dot{A}_n)^{-\frac{1}{2}}\cdot\|_{l_n^2}$ produces
	\begin{align}\label{eq2}
		&\quad\sqrt{\frac{n}{\pi}}\LL\Sigma_n(V(t))R(t), (-\dot{A}_n)^{-1}V(t)\RR_{l_n^2}\le 	\sqrt{\frac{n}{\pi}}\|\Sigma_n(V(t))R(t)\|_{\l_n^2} \|(-\dot{A}_n)^{-1}V(t)\|_{\l_n^2}\notag\\
		&\le K|R(t)|\|(-\dot{A}_n)^{-\frac{1}{2}}V(t)\|_{\l_n^2}\le K|R(t)|^2+K\|(-\dot{A}_n)^{-\frac{1}{2}}V(t)\|_{\l_n^2}^2.
	\end{align}
	
	Taking the inner product $\LL\cdot,(-\dot{A}_n)^{-1}V(t)\RR_{l_n^2}$ on both sides of \eqref{SDE} and using \eqref{eq3}-\eqref{eq2}, we obtain that for any $\eta\in(0,1)$,
	\begin{align*}
		&\frac{1}{2}\;\frac{\ud }{\ud t}\|(-\dot{A}_n)^{-\frac{1}{2}}V(t)\|_{\l_n^2}^2=\LL \dot{V}(t),(-\dot{A}_n)^{-1}V(t)\RR_{l_n^2}\\
		=&\;\LL-A_n^2V(t),(-\dot{A}_n)^{-1}V(t)\RR_{l_n^2}+\LL A_nB_n(V(t)),(-\dot{A}_n)^{-1}V(t)\RR_{l_n^2}+\sqrt{\frac{n}{\pi}}\LL\Sigma_n(V(t))R(t), (-\dot{A}_n)^{-1}V(t)\RR_{l_n^2}\\
		\le &\;-\|(-A_n)^{\frac{1}{2}}V(t)\|_{l_n^2}^2-(1-\eta)\|V(t)\|_{l_n^4}^4+K(\eta,a,T)\big(1+|R(t)|^2+\|(-\dot{A}_n)^{-\frac{1}{2}}V(t)\|_{\l_n^2}^2\big).
	\end{align*}
	Taking $\eta=\frac{1}{2}$ in the above formula and using $\|R\|_{\mbf L^2(0,T)}\le a$, one has
	\begin{align*}
		&\;\frac{1}{2}\|(-\dot{A}_n)^{-\frac{1}{2}}V(t)\|_{\l_n^2}^2+\int_0^t\|(-A_n)^{\frac{1}{2}}V(s)\|_{l_n^2}^2\ud s+\frac{1}{2}\int_0^t\|V(s)\|_{l_n^4}^4\ud s\\
		\le&\;\frac{1}{2}\|(-\dot{A}_n)^{-\frac{1}{2}}U^n(0)\|_{l_n^2}^2+K(a,T)\int_0^t\|(-\dot{A}_n)^{-\frac{1}{2}}V(s)\|_{\l_n^2}^2\ud s+K(a,T).
	\end{align*}
	Then using $\|(-\dot{A}_n)^{-\frac{1}{2}}U^n(0)\|_{l_n^2}\le\|U^n(0)\|_{l_n^\infty}\le \|u_0\|_{\mbf C(\OO)}$ yields
	\begin{align}\label{eq4}
		&\;\|(-\dot{A}_n)^{-\frac{1}{2}}V(t)\|_{\l_n^2}^2+2\int_0^t\|(-A_n)^{\frac{1}{2}}V(s)\|_{l_n^2}^2\ud s+\int_0^t\|V(s)\|_{l_n^4}^4\ud s\nonumber\\
		\le&\;K(a,T)+K(a,T)\int_0^t\|(-\dot{A}_n)^{-\frac{1}{2}}V(s)\|_{\l_n^2}^2\ud s.
	\end{align} 
	By the Gronwall inequality, it holds that  $\sup\limits_{t\in[0,T]}\|(-\dot{A}_n)^{-\frac{1}{2}}V(t)\|_{\l_n^2}\le K(a,T)$, which along with \eqref{eq4} gives
	\begin{align*}
		\int_0^T\|(-A_n)^{\frac{1}{2}}V(t)\|_{l_n^2}^2\ud t+\int_0^T\|V(t)\|_{l_n^4}^4\ud t\le K(a,T).
	\end{align*}
	
	\textbf{Step $2$: we prove 
		\begin{align}\label{VH2}
			\sup_{t\in[0,T]}\|V(t)\|_{l_n^2}+\int_0^T\|A_nV(t)\|_{l_n^2}^2\ud t\le K(a,T).
		\end{align}
	}
	
	Taking the inner product $\LL\cdot,V(t)\RR$ on both sides of \eqref{SDE} yields
	\begin{align*}
		\frac{1}{2}\frac{\ud }{\ud t}\|V(t)\|_{l_n^2}+\|A_nV(t)\|_{l_n^2}^2=\LL A_nB_n(V(t)),V(t)\RR_{l_n^2}+\sqrt{\frac{n}{\pi}}\LL\Sigma_n(V(t))R(t),V(t)\RR_{l_n^2}.
	\end{align*}
	It follows from the above formula, \eqref{xBf}, \eqref{lnnorm},  \eqref{Rt} and the Young inequality  that
	\begin{align*}
		&\;	\frac{1}{2}\frac{\ud }{\ud t}\|V(t)\|^2_{l_n^2}+\|A_nV(t)\|_{l_n^2}^2\le -\LL A_nV(t),V(t)\RR_{l_n^2}+K|R(t)|\|V(t)\|_{l_n^2}\nonumber\\
		\le &\;\|(-A_n)^{\frac12}V(t)\|_{l_n^2}^2+K|R(t)|^2+K\|V(t)\|_{l_n^2}^2,\\
		\le &\;\|(-A_n)^{\frac12}V(t)\|_{l_n^2}^2+K|R(t)|^2+K\|V(t)\|_{l_n^4}^4+K.
	\end{align*}
	Combining the above formula and \eqref{VH1}, we have that for ant $t\in[0,T]$,
	\begin{align*}
		\|V(t)\|^2_{l_n^2}+2\int_0^t\|A_nV(s)\|_{l_n^2}^2\ud s&\le \|U^n(0)\|_{l_n^2}^2+2\int_0^T\|(-A_n)^{\frac12}V(t)\|_{l_n^2}^2\ud t+K\int_0^T\|V(t)\|_{l_n^4}^4\ud t+K(a,T)\le K(a,T),
	\end{align*}
	where we have used $\|U^n(0)\|_{l_n^\infty}\le \|u_0\|_{\mbf C(\OO)}$.
	
	\textbf{Step $3$: we prove $\sup\limits_{t\in[0,T]}\|V(t)\|_{l_n^6}\le K(a,T)$.
	}
	
	Using the variation of constant formula, we derive from \eqref{SDE} that
	\begin{align}\label{Vt}
		V(t)=&\;e^{-A_n^2t}U^n(0)+\int_0^te^{-A_n^2(t-s)}A_nB_n(V(s))\ud s\nonumber\\
		&\;+\sqrt{\frac{n}{\pi}}\int_{0}^te^{-A_n^2(t-s)}\Sigma_n(V(s))R(s)\ud s,~t\in[0,T].
	\end{align}
	Thus, 
	\begin{align}\label{eq5}
		\|V(t)\|_{l_n^6}\le&\; \|e^{-A_n^2t}U^n(0)\|_{l_n^6}+\int_{0}^t\|e^{-A_n^2(t-s)}A_nB_n(V(s))\|_{l_n^6}\ud s\nonumber\\
		&\;+\sqrt{\frac{n}{\pi}}\int_{0}^t\|e^{-A_n^2(t-s)}\Sigma_n(V(s))R(s)\|_{l_n^6}\ud s.
	\end{align}
	A direct computation leads to
	\begin{align}\label{Un0H1}
		\|(-A_n)^{\frac12}U^n(0)\|_{l_n^2}^2=\frac{n}{\pi}\sum_{j=1}^{n-1}|u_0(x_{j+1})-u_0(x_j)|^2\le K\|u_0'\|^2_{\mbf C(\OO)},
	\end{align}
	which along with \eqref{l2H1} and $|e^{-A_n^2t}|\le 1$ yields
	\begin{align}\label{Un0}
		&\phantom{\le}\|e^{-A_n^2t}U^n(0)\|_{l_n^6}\le \|e^{-A_n^2t}U^n(0)\|_{l_n^\infty}\nonumber\\
		&\le K\|e^{-A_n^2t}U^n(0)\|_{l_n^2}+	\|e^{-A_n^2t}(-A_n)^{\frac12}U^n(0)\|_{l_n^2}\le K.
	\end{align}
	By the spectrum mapping theorem and the symmetry of $ A_n$, we see that for $\gamma>0$,
	\begin{equation}\label{smooth}
		|e^{-\frac{1}{2} A_n^2(t-s)}(- A_n)^{\gamma}|=\max_{1\le j\le n-1}e^{-\frac{1}{2}\lambda_{j,n}^2(t-s)}(-\lambda_{j,n})^\gamma\le C(t-s)^{-\frac{\gamma}{2}},
	\end{equation}
	since $x\mapsto x^\gamma e^{-x^2}$ is uniformly bounded on $[0,\infty)$. 
	It follows from \eqref{interpolation} and \eqref{smooth} that
	\begin{align}\label{eq6}
		&\;\int_{0}^t\|e^{-A_n^2(t-s)}A_nB_n(V(s))\|_{l_n^6}\ud s=\int_{0}^t\|e^{-\frac{1}{2}A_n^2(t-s)}e^{-\frac{1}{2}A_n^2(t-s)}A_nB_n(V(s))\|_{l_n^6}\ud s\nonumber\\
		\le &\; K\int_0^t(1+(t-s)^{-\frac{1}{12}})\|e^{-\frac{1}{2}A_n^2(t-s)}A_nB_n(V(s))\|_{l_n^2}\ud s\nonumber\\
		\le& \;K(T)\int_0^t(t-s)^{-\frac{7}{12}}\|B_n(V(s))\|_{l_n^2}\ud s\nonumber\\
		\le &\;K(T)\int_0^t(t-s)^{-\frac{7}{12}}(1+\|V(s)\|_{l_n^6}^3)\ud s,
	\end{align}
	where we have used the fact $1\le K(T)(t-s)^{-\frac{1}{12}}$ for $0\le s<t\le T$. Further, applying \eqref{l6h2} and \eqref{VH2} gives that for any $t\in[0,T]$,
	\begin{align}\label{eq7}
		\|V(t)\|_{l_n^6}\le K\big(\|A_nV(t)\|_{l_n^2}^{\frac16}\|V(t)\|_{l_n^2}^{\frac56}+\|V(t)\|_{l_n^2}\big)\le K(a,T)\big(1+\|A_nV(t)\|_{l_n^2}^{\frac16}\big).
	\end{align}
	Substituting \eqref{eq7} into \eqref{eq6} and using the H\"older inequality and \eqref{VH2}, we arrive at
	\begin{align}\label{eq8}
		&\;\int_{0}^t\|e^{-A_n^2(t-s)}A_nB_n(V(s))\|_{l_n^6}\ud s\le K(T)\int_0^t(t-s)^{-\frac{7}{12}}\big(1+\|A_nV(s)\|_{l_n^2}^\frac{1}{2}\big)\ud s\nonumber\\
		\le &\;K(T)\bigg(\int_0^t(t-s)^{-\frac{7}{9}}\ud s\bigg)^{\frac34}\bigg(\int_0^t\big(1+\|A_nV(s)\|_{l_n^2}^{2}\big)\ud s\bigg)^{\frac14}\le K(a,T).
	\end{align}
	By \eqref{interpolation}, \eqref{Rt} and the H\"older inequality,
	\begin{align}\label{eq9}
		&\;\sqrt{\frac{n}{\pi}}\int_{0}^t\|e^{-A_n^2(t-s)}\Sigma_n(V(s))R(s)\|_{l_n^6}\ud s\nonumber\
		\le K\sqrt{\frac{n}{\pi}}	\int_{0}^t(1+(t-s)^{-\frac{1}{12}})\|\Sigma_n(V(s))R(s)\|_{l_n^2}\ud s\nonumber\\
		\le&\;K	\int_{0}^t(1+(t-s)^{-\frac{1}{12}})|R(s)|\ud s\le K\big(\int_{0}^t(1+(t-s)^{-\frac{1}{6}})\ud s\big)^{\frac{1}{2}}\|R\|_{\mbf L^2(0,T)}\le K(a,T).
	\end{align}
	Substituting \eqref{Un0}, \eqref{eq8} and \eqref{eq9} into \eqref{eq5}, we obtain $\sup\limits_{t\in[0,T]}\|V(t)\|_{l_n^6}\le K(a,T)$.
	
	\textbf{Step $4$: we prove $\sup\limits_{t\in[0,T]}\|V(t)\|_{l_n^\infty}\le K(a,T)$.
	}
	By \eqref{Rt}, \eqref{Vt}, \eqref{Un0H1}, \eqref{smooth} and the result of the third step,
	\begin{align*}
		\|(-A_n)^{\frac12}V(t)\|_{l_n^2}&\le \|e^{-A_n^2t}(-A_n)^{\frac12}U^n(0)\|_{l_n^2}+\int_0^t\|e^{-A_n^2(t-s)}(-A_n)^{\frac12}A_nB_n(V(s))\|_{l_n^2}\ud s\\
		&\phantom{\le}+ \sqrt{\frac{n}{\pi}}\int_{0}^t\|e^{-A_n^2(t-s)}(-A_n)^{\frac12}\Sigma_n(V(s))R(s)\|_{l_n^2}\ud s\\
		&\le K+K\int_0^t(t-s)^{-\frac34}\|B_n(V(s))\|_{l_n^2}\ud s+K\int_0^t(t-s)^{-\frac14}|R(s)|\ud s\\
		&\le K+K\int_0^t(t-s)^{-\frac34}\big(1+\|V(s)\|_{l_n^6}^3\big)\ud s+K\big(\int_0^t(t-s)^{-\frac12}\ud s\big)^{\frac12}\|R\|_{\mbf L^2(0,T)}\\
		&\le K(a,T).
	\end{align*}
	Finally combining the above estimate,  \eqref{l2H1} and \eqref{VH2}, we have $\sup\limits_{t\in[0,T]}\|V(t)\|_{l_n^\infty}\le K(a,T)$. Thus the proof is complete. 
\end{proof}

\bibliographystyle{plain}
\bibliography{mybibfile}

\begin{thebibliography}{10}

\bibitem{LDPofCH}
L.~Boulanba and M.~Mellouk.
\newblock Large deviations for a stochastic {C}ahn-{H}illiard equation in
  {H}\"{o}lder norm.
\newblock {\em Infin. Dimens. Anal. Quantum Probab. Relat. Top.},
  23(2):2050010, 17, 2020.

\bibitem{Dupuis08}
A.~Budhiraja, P.~Dupuis, and V.~Maroulas.
\newblock Large deviations for infinite dimensional stochastic dynamical
  systems.
\newblock {\em Ann. Probab.}, 36(4):1390--1420, 2008.

\bibitem{CH58}
J.~W Cahn and J.~E Hilliard.
\newblock Free energy of a nonuniform system. i. interfacial free energy.
\newblock {\em The Journal of chemical physics}, 28(2):258--267, 1958.

\bibitem{Cardon2001}
C.~Cardon-Weber.
\newblock Cahn-{H}illiard stochastic equation: existence of the solution and of
  its density.
\newblock {\em Bernoulli}, 7(5):777--816, 2001.

\bibitem{CCZZ18}
S.~Chai, Y.~Cao, Y.~Zou, and W.~Zhao.
\newblock Conforming finite element methods for the stochastic
  {C}ahn--{H}illiard--{C}ook equation.
\newblock {\em Appl. Numer. Math.}, 124:44--56, 2018.

\bibitem{ChenCC}
C.~Chen.
\newblock A symplectic discontinuous {G}alerkin full discretization for
  stochastic {M}axwell equations.
\newblock {\em SIAM J. Numer. Anal.}, 59(4):2197--2217, 2021.

\bibitem{LDPofInvariant}
C.~Chen, Z.~Chen, J.~Hong, and D.~Jin.
\newblock Large deviations principles of sample paths and invariant measures of
  numerical methods for parabolic {SPDEs}.
\newblock {\em arXiv:2106.11018}, 2021.

\bibitem{LDPosc}
C.~Chen, J.~Hong, D.~Jin, and L.~Sun.
\newblock Asymptotically-preserving large deviations principles by stochastic
  symplectic methods for a linear stochastic oscillator.
\newblock {\em SIAM J. Numer. Anal.}, 59(1):32--59, 2021.

\bibitem{LDPxde}
C.~Chen, J.~Hong, D.~Jin, and L.~Sun.
\newblock Large deviations principles for symplectic discretizations of
  stochastic linear {S}chr\"odinger equation.
\newblock {\em Potential Anal.}, pages
  https://doi.org/10.1007/s11118--022--09990--z, 2022.

\bibitem{CH20}
J.~Cui and J.~Hong.
\newblock Absolute continuity and numerical approximation of stochastic
  {C}ahn--{H}illiard equation with unbounded noise diffusion.
\newblock {\em J. Differential Equations}, 269(11):10143--10180, 2020.

\bibitem{CH22}
J.~Cui and J.~Hong.
\newblock Wellposedness and regularity estimates for stochastic
  {C}ahn--{H}illiard equation with unbounded noise diffusion.
\newblock {\em Stoch. Partial Differ. Equ. Anal. Comput.}, pages 1--37, 2022.

\bibitem{CHS21}
J.~Cui, J.~Hong, and L.~Sun.
\newblock Strong convergence of full discretization for stochastic
  {C}ahn--{H}illiard equation driven by additive noise.
\newblock {\em SIAM J. Numer. Anal.}, 59(6):2866--2899, 2021.

\bibitem{Gamma93}
G.~Dal~Maso.
\newblock {\em An {I}ntroduction to {$\Gamma$}-Convergence}, volume~8 of {\em
  Progress in Nonlinear Differential Equations and their Applications}.
\newblock Birkh\"{a}user Boston, Inc., Boston, MA, 1993.

\bibitem{Dembo}
A.~Dembo and O.~Zeitouni.
\newblock {\em Large {D}eviations {T}echniques and {A}pplications}, volume~38
  of {\em Stochastic Modelling and Applied Probability}.
\newblock Springer-Verlag, Berlin, 2010.

\bibitem{EL92}
C.~M. Elliott and S.~Larsson.
\newblock Error estimates with smooth and nonsmooth data for a finite element
  method for the {C}ahn-{H}illiard equation.
\newblock {\em Math. Comp.}, 58(198):603--630, S33--S36, 1992.

\bibitem{FLZ20}
X.~Feng, Y.~Li, and Y.~Zhang.
\newblock A fully discrete mixed finite element method for the stochastic
  {C}ahn-{H}illiard equation with gradient-type multiplicative noise.
\newblock {\em J. Sci. Comput.}, 83(1):Paper No. 23, 24, 2020.

\bibitem{FW}
M.~I. Freidlin and A.~D. Wentzell.
\newblock {\em Random {P}erturbations of {D}ynamical {S}ystems}, volume 260 of
  {\em Grundlehren der mathematischen Wissenschaften [Fundamental Principles of
  Mathematical Sciences]}.
\newblock Springer-Verlag, New York, second edition, 1998.
\newblock Translated from the 1979 Russian original by Joseph Sz\"{u}cs.

\bibitem{FKLL18}
D.~Furihata, M.~Kov\'{a}cs, S.~Larsson, and F.~Lindgren.
\newblock Strong convergence of a fully discrete finite element approximation
  of the stochastic {C}ahn--{H}illiard equation.
\newblock {\em SIAM J. Numer. Anal.}, 56(2):708--731, 2018.

\bibitem{GI98}
I.~Gy\"{o}ngy.
\newblock Lattice approximations for stochastic quasi-linear parabolic partial
  differential equations driven by space-time white noise. {I}.
\newblock {\em Potential Anal.}, 9(1):1--25, 1998.

\bibitem{LDPofonepoint}
J.~Hong, D.~Jin, and D.~Sheng.
\newblock Numerical approximations of one-point large deviations rate functions
  of stochastic differential equations with small noise.
\newblock {\em arXiv:2102.04061}, 2021.

\bibitem{Sheng}
J.~Hong, D.~Jin, and D.~Sheng.
\newblock Convergence analysis of a finite difference method for stochastic
  {C}ahn--{H}illiard equation.
\newblock {\em arXiv:2202.09055}, 2022.

\bibitem{Sheng1}
J.~Hong, D.~Jin, and D.~Sheng.
\newblock Finite difference method for stochastic {C}ahn--{H}illiard equation:
  Strong convergence rate and density convergence.
\newblock {\em arXiv:2203.00571}, 2022.

\bibitem{LDPlan}
J.~Hong, D.~Jin, D.~Sheng, and L.~Sun.
\newblock Numerically asymptotical preservation of the large deviations
  principles for invariant measures of {L}angevin equations.
\newblock {\em arXiv:2009.13336}, 2020.

\bibitem{LDPnonLip}
Y.~Hu and G.~Lan.
\newblock Large deviation principle of {SDE}s with non-{L}ipschitzian
  coefficients under localized conditions.
\newblock {\em arXiv:1404.1481}, 2014.

\bibitem{LDPofSWE}
D.~Jin, J.~Hong, and D.~Sheng.
\newblock Convergence analysis of one-point large deviations rate functions of
  numerical discretizations for stochastic wave equations with small noise.
\newblock {\em arXiv:2209.08341}, 2022.

\bibitem{KD14}
D.~Khoshnevisan.
\newblock {\em Analysis of {S}tochastic {P}artial {D}ifferential {E}quations},
  volume 119 of {\em CBMS Regional Conference Series in Mathematics}.
\newblock Published for the Conference Board of the Mathematical Sciences,
  Washington, DC; by the American Mathematical Society, Providence, RI, 2014.

\bibitem{KLM11}
M.~Kov\'{a}cs, S.~Larsson, and A.~Mesforush.
\newblock Finite element approximation of the {C}ahn--{H}illiard--{C}ook
  equation.
\newblock {\em SIAM J. Numer. Anal.}, 49(6):2407--2429, 2011.

\bibitem{LM11}
S.~Larsson and A.~Mesforush.
\newblock Finite-element approximation of the linearized
  {C}ahn--{H}illiard--{C}ook equation.
\newblock {\em IMA J. Numer. Anal.}, 31(4):1315--1333, 2011.

\bibitem{NCS84}
A.~Novick-Cohen and L.~A. Segel.
\newblock Nonlinear aspects of the {C}ahn--{H}illiard equation.
\newblock {\em Phys. D}, 10(3):277--298, 1984.

\bibitem{PR07}
Claudia Pr\'{e}v\^{o}t and Michael R\"{o}ckner.
\newblock {\em A concise course on stochastic partial differential equations},
  volume 1905 of {\em Lecture Notes in Mathematics}.
\newblock Springer, Berlin, 2007.

\bibitem{QW20}
R.~Qi and X.~Wang.
\newblock Error estimates of semidiscrete and fully discrete finite element
  methods for the {C}ahn--{H}illiard--{C}ook equation.
\newblock {\em SIAM J. Numer. Anal.}, 58(3):1613--1653, 2020.

\bibitem{Gamma18}
F.~Rindler.
\newblock {\em Calculus of variations}.
\newblock Universitext. Springer, Cham, 2018.

\bibitem{LDPCH09}
K.~Shi, D.~Tang, and Y.~Wang.
\newblock Large deviation for stochastic {C}ahn--{H}illiard partial
  differential equations.
\newblock {\em Acta Math. Sin. (Engl. Ser.)}, 25(7):1157--1174, 2009.

\bibitem{ZL22}
L.~Zhou and Y.~Li.
\newblock An {LDG} method for stochastic {C}ahn-{H}illiard type equation driven
  by general multiplicative noise involving second-order derivative.
\newblock {\em Commun. Comput. Phys.}, 31(2):516--547, 2022.

\bibitem{ZG18}
G.~E. Zouraris.
\newblock An {IMEX} finite element method for a linearized
  {C}ahn--{H}illiard--{C}ook equation driven by the space derivative of a
  space-time white noise.
\newblock {\em Comput. Appl. Math.}, 37(5):5555--5575, 2018.

\end{thebibliography}

\end{document}